\documentclass[a4paper,11pt,reqno]{amsart}
\usepackage[margin=2cm,includehead,includefoot]{geometry}
\usepackage{amsmath}
\usepackage{amssymb}
\usepackage{amsthm}
\usepackage{nccmath}
\usepackage[english]{babel}
\usepackage[hidelinks,pagebackref]{hyperref}
\usepackage{mathrsfs}
\usepackage{tikz}
\usepackage{xspace}
\usepackage[capitalise,noabbrev,english]{cleveref}
\usepackage{stmaryrd}
\usepackage{genyoungtabtikz}
\usepackage{needspace}
\usepackage{enumitem}

\numberwithin{equation}{section}

\newtheoremstyle{mattthm}{}{}{\slshape}{}{\bfseries}{.}{ }{}
\theoremstyle{mattthm}

\newtheorem{theor}[equation]{Theorem}
\newtheorem{lemma}[equation]{Lemma}
\newtheorem{cor}[equation]{Corollary}
\newtheorem{propn}[equation]{Proposition}
\newtheorem{conj}[equation]{Conjecture}

\newtheoremstyle{mattdef}{}{}{}{}{\bfseries}{.}{ }{}

\theoremstyle{mattdef}

\newtheorem{rem}[equation]{Remark}

\crefname{thm}{Theorem}{Theorems}
\crefname{theor}{Theorem}{Theorems}
\crefname{propn}{Proposition}{Propositions}
\crefname{lemma}{Lemma}{Lemmas}
\crefname{cor}{Corollary}{Corollaries}
\crefname{conj}{Conjecture}{Conjectures}
\Crefname{thm}{Theorem}{Theorems}
\Crefname{theor}{Theorem}{Theorems}
\Crefname{propn}{Proposition}{Propositions}
\Crefname{lemma}{Lemma}{Lemmas}
\Crefname{cor}{Corollary}{Corollaries}
\Crefname{conj}{Conjecture}{Conjectures}

\newcommand\nequiv{\not\equiv}
\renewcommand\iff{if and only if\xspace}
\newcommand{\clam}{\begin{description}\item[\hspace{\leftmargin}Claim]}
\newcommand{\prof}{\item[\hspace{\leftmargin}Proof]}
\newcommand{\malc}{\end{description}}
\newcommand\lset[2]{\left\{\left.#1\ \right|\ \smash{#2}\right\}}
\newcommand\rset[2]{\left\{\smash{#1}\ \left|\ #2\right.\right\}}
\newcommand\df{\stackrel3{\dots.}}
\newcommand\calt{\mathcal T}
\newcommand\calc{\mathcal C}
\newcommand\caltc{\mathcal{TC}}
\newcommand\calb{\mathcal B}
\newcommand\calbc{\mathcal{BC}}
\newcommand\nchar[1]{\operatorname{char}(#1)}
\newcommand\ol\overline
\newcommand\ppmod[1]{\ (\operatorname{mod}\,#1)}
\newcommand\str[2]{\operatorname{c}_{#1}(#2)}
\newcommand\zz[2]{\operatorname{d}_{#1}(#2)}
\newcommand\dzz[2]{\operatorname{b}_{#1}(#2)}
\newcommand\lad[2]{\operatorname{lad}_{#1}(#2)}
\renewcommand\rem[2]{\operatorname{srem}_{#1}(#2)}
\newcommand\add[2]{\operatorname{sadd}_{#1}(#2)}
\newcommand\badd[2]{\operatorname{add}_{#1}(#2)}
\newcommand\brem[2]{\operatorname{rem}_{#1}(#2)}
\newcommand\bbn{\mathbb N}
\newcommand\bbc{\mathbb C}
\newcommand\bbf{\mathbb F}
\newcommand\ladd[1]{\mathcal L_{#1}}
\newcommand{\reg}{\mathrm{reg}}
\newcommand{\Ind}{\mathrm{f}}
\newcommand{\Res}{\mathrm{e}}
\newcommand{\sym}{\mathfrak S}
\newcommand{\hsym}{\hat{\mathfrak S}}
\newcommand{\alt}{\mathfrak A}
\newcommand{\halt}{\hat{\mathfrak A}}
\newcommand\hf{.5}
\newcommand\im{^{\triangledown i}}
\newcommand\ipi[1]{^{\vartriangle #1}}
\newcommand\ip{\ipi i}
\newcommand\adbl{addable\xspace}
\newcommand\rmbl{removable\xspace}
\newcommand\sadbl{strictly-addable\xspace}
\newcommand\srmbl{strictly-removable\xspace}

\newcounter{casecount}0
\newcounter{subcasecount}
\newcounter{subsubcasecount}

\newcommand\case{\addtocounter{casecount}1\setcounter{subcasecount}0\def\caselabel{Case \arabic{casecount}}\subsubsection*{\bf\caselabel}}
\newcommand\subcase{\addtocounter{subcasecount}1\setcounter{subsubcasecount}0\def\subcaselabel{Case \arabic{casecount}.\arabic{subcasecount}}\subsubsection*{\bf\subcaselabel}}
\newcommand\subsubcase{\addtocounter{subsubcasecount}1\def\subsubcaselabel{Case \arabic{casecount}.\arabic{subcasecount}.\arabic{subsubcasecount}}\subsubsection*{\bf\subsubcaselabel}}

\def\Md#1{\text{ }(\text{\rm mod } #1)\,}
\newcommand{\Z}{\mathbb{Z}}
\newcommand\bbz{\mathbb Z}
\newcommand{\F}{\mathbb{F}}
\newcommand{\sgn}{\mathbf{\mathrm{sgn}}}
\newcommand{\la}{\lambda}
\newcommand{\La}{\Lambda}
\newcommand{\be}{\beta}
\newcommand{\al}{\alpha}
\newcommand{\eps}{\epsilon}
\newcommand{\ga}{\gamma}
\newcommand{\de}{\delta}
\newcommand\si\sigma
\newcommand{\wrs}{\wr_{\text{s}}}
\newcommand\awrd{A\wrs\sym_d}
\newcommand{\heps}{{\hat{\eps}}}
\newcommand{\hphi}{{\hat{\phi}}}
\newcommand{\te}{\tilde{\mathrm e}}
\newcommand{\tf}{\tilde{\mathrm f}}
\newcommand{\da}{{\downarrow}}
\newcommand{\ua}{{\uparrow}}
\newcommand\dalt{\da_{\halt_n}}

\def\RP{{\mathcal{P}_0}}
\def\rpp{\mathcal{RP}_p}
\def\Par{{\mathcal{P}}}
\def\s{{\mathbf{S}}}
\def\ss{{\mathrm{S}}}
\def\st{{\mathrm{T}}}
\def\D{{\mathbf{D}}}
\def\P{{\mathbf{P}}}
\newcommand\ls{\DOTSB\leqslant}
\newcommand\gs{\DOTSB\geqslant}
\renewcommand\leq{\DOTSB\leqslant}
\renewcommand\geq{\DOTSB\geqslant}

\renewcommand\ge{\DOTSB\geqslant}
\newcommand\card[1]{|#1|}
\newcommand\ep[2]{\epsilon_{#1}(#2)}
\newcommand\hep[2]{\hat{\epsilon}_{#1}(#2)}
\newcommand\ddeg[1]{\operatorname{ddim}(#1)}

\newcommand{\doms}{\vartriangleright}
\newcommand{\domby}{\trianglelefteqslant}
\newcommand{\domsby}{\vartriangleleft}
\renewcommand\lhd\domsby
\renewcommand\rhd\doms
\newcommand\dn[2]{\hep{#1}{#2}>\ep{#1}{#2^\reg}}
\newcommand\spe[1]{\mathscr S^{#1}}
\newcommand\ospe[1]{\ol{\mathscr S}^{#1}}
\newcommand\jms[1]{\mathscr D^{#1}}
\newcommand\spc{special\xspace}
\newcommand\Spc{Special\xspace}

\setitemize[1]{label={$\diamond$}}

\newcommand\con{cited~on~p.~}
\newcommand\cons{cited~on~pp.~}

\renewcommand*{\backref}[1]{}
\renewcommand*{\backrefalt}[4]{\ifcase #1 \hspace*{\fill}{\small [no~citations.]}\or\hspace*{\fill}{\small [\con#2]}\else\hspace*{\fill}{\small [\cons#2]}\fi}
\newcommand\bkp{\backrefprint\renewcommand\con{}\renewcommand\cons{}}

\makeatletter
\renewenvironment{proof}[1][\proofname] {\par\pushQED{\qed}\normalfont\topsep6\p@\@plus6\p@\relax\trivlist\item[\hskip\labelsep\bfseries#1\@addpunct{.}]\ignorespaces}{\popQED\endtrivlist\@endpefalse}
\makeatother
\begin{document}

\title[Irreducible spin representations in characteristic $3$]{On the irreducible spin representations of\\ symmetric and alternating groups which\\ remain irreducible in characteristic~$3$}

\author{Matthew Fayers}
\address{Queen Mary University of London\\Mile End Road\\London E1 4NS\\U.K.}
\email{m.fayers@qmul.ac.uk}

\author{Lucia Morotti}
\address
 {Leibniz Universit\"{a}t Hannover\\ Institut f\"{u}r Algebra, Zahlentheorie und Diskrete Mathematik\\ 30167 Hannover\\ Germany} 
 \email{morotti@math.uni-hannover.de}

\begin{abstract}
For any finite group $G$ and any prime $p$ one can ask which ordinary irreducible representations remain irreducible in characteristic $p$, or more generally, which representations remain homogeneous in characteristic $p$. In this paper we address this question when $G$ is a proper double cover of the symmetric or alternating group. We obtain a classification when $p=3$ except in the case of a certain family of partitions relating to spin RoCK blocks. Our techniques involve induction and restriction, degree calculations, decomposing projective characters and recent results of Kleshchev and Livesey on spin RoCK blocks.
\end{abstract}

\maketitle

\vspace{12pt}

\tableofcontents

\section{Introduction}

In the representation theory of finite groups it is interesting and useful to classify the ordinary irreducible representations of a group $G$ that remain irreducible modulo a prime $p$. This classification has been carried out for various families of groups, including the symmetric groups \cite{jm2,slred,mfred,mfirred}, the alternating groups \cite{F2,F3} and the finite general linear and special linear groups \cite{jmjs,kt} in non-defining characteristic, as well as the double covers of the symmetric and alternating groups when $p=2$ \cite{Fs,F}. In this paper we address the same problem for the double covers in odd characteristic, obtaining an almost-complete classification in the case $p=3$.

In fact in this setting it is more natural to classify representations which remain \emph{homogeneous} (that is, with all composition factors isomorphic) modulo $p$. To state our main theorem, let $\hsym_n$ denote a proper double cover of the symmetric group $\sym_n$, and $\halt_n$ the corresponding double cover of the alternating group $\alt_n$. Irreducible representations of these groups on which the central involution $z$ acts trivially correspond to representations of the quotient groups $\sym_n$ and $\alt_n$, so we concentrate on representations on which $z$ acts non-trivially, also called \emph{spin} representations.  The spin representations of $\bbc\hsym_n$ are modules for the twisted group algebra $\calt_n$, which is usually studied as a superalgebra, with irreducible supermodules labelled by strict partitions of $n$. We let $\s^\la_0$ be the irreducible supermodule labelled by a strict partition $\la$, and let $\s^\la_p$ be a reduction modulo $p$ of $\s^\la_0$. Our main result is the following.

\begin{theor}\label{mainthm}
If $\la$ is a strict partition and $\s^\la_3$ is homogeneous as a supermodule then one of the following holds:
\begin{enumerate}
\item\label{qs}
$\la=\nu+3\al$ where $\nu$ is a $3$-bar core and $\al$ is a partition with $l(\al)\leq l(\nu)$;
\item\label{basicsp}
$\la=(3a)$ with $a\geq 2$;
\item
$\la=\nu\sqcup(3)$ with $\nu$ a $3$-bar core;
\item\label{smallcases}
$\la$ is one of the partitions $(2,1)$, $(3,2,1)$, $(4,3,2)$, $(4,3,2,1)$, $(5,3,2,1)$, $(5,4,3,1)$, $(5,4,3,2)$, $(5,4,3,2,1)$, $(7,4,3,2,1)$, $(8,5,3,2,1)$.
\end{enumerate}
Furthermore in cases (\ref{basicsp})--(\ref{smallcases}) $\s^\la_3$ is homogeneous.
\end{theor}

This theorem gives a classification of supermodules which are homogeneous modulo $3$, except for partitions appearing in case (\ref{qs}). For this case we make a precise conjecture (\cref{spcconj} below) saying exactly which $\s^\la$ are homogeneous, and we hope to resolve this in future work.

A classification of $\calt_n$-supermodules which reduce homogeneously directly leads to a classification of spin representations for $\hsym_n$ and $\halt_n$ which remain irreducible in characteristic $p$ (see \cref{irredhomog}). From our almost-complete classification for $p=3$ in \cref{mainthm}, we can deduce the following. Recall that a strict partition $\la$ is \emph{even} if it has an even number of positive even parts, and \emph{odd} otherwise. For each even strict partition $\la$ of $n$ we have an irreducible $\bbc\hsym_n$-module $\ss^\la$, and a pair of irreducible $\bbc\halt_n$-modules $\st^{\la,\pm}$. For each odd strict partition $\la$ of $n$, we have a pair of irreducible $\bbc\hsym_n$-modules $\ss^{\la,\pm}$ and an irreducible $\bbc\halt_n$-module $\st^\la$. All irreducible spin modules for $\bbc\hsym_n$ and $\bbc\halt_n$ arise in this way.

\begin{theor}\label{mainmodule}
Suppose $M$ is an irreducible spin module for $\bbc\hsym_n$ or $\bbc\halt_n$. If the reduction of $M$ modulo $3$ is irreducible, then one of the following holds:
\begin{enumerate}
\item
$M$ is one of the modules $\ss^{\la(,\pm)}$ or $\st^{\la(,\pm)}$, where $\nu$ is a $3$-bar core and $\al$ is a partition with $l(\al)\leq l(\nu)$ and $\la=\nu+3\al$;
\item\label{basicspmod}
$M=\ss^{(6a),\pm}$ or $M=\st^{(6a-3),\pm}$ for $a\gs1$;
\item
$M=\ss^{\nu\sqcup(3),\pm}$, where $\nu$ is an odd $3$-bar core;
\item
$M=\st^{\nu\sqcup(3),\pm}$, where $\nu$ is an even $3$-bar core;
\item\label{smallcasesmod}
$M$ is one of $\ss^{(2,1),\pm}$, $\st^{(2,1)}$, $\ss^{(3,2,1),\pm}$, $\st^{(4,3,2),\pm}$, $\st^{(4,3,2,1),\pm}$, $\ss^{(5,3,2,1),\pm}$, $\ss^{(5,4,3,1),\pm}$, $\st^{(5,4,3,2),\pm}$, $\st^{(5,4,3,2,1),\pm}$, $\st^{(7,4,3,2,1),\pm}$, $\st^{(8,5,3,2,1),\pm}$.
\end{enumerate}
Furthermore in cases (\ref{basicspmod}--\ref{smallcasesmod}) $M$ is irreducible modulo $3$.
\end{theor}

Our approach to the representation theory of $\hsym_n$ follows the Brundan--Kleshchev approach to double covers via supermodules, using in particular their branching rules for irreducible supermodules and their analogue of James's regularisation theorem. We also make use of the very recent study by Kleshchev and Livesey of RoCK blocks for double covers of symmetric groups; this leads us to an analysis of modules for a certain wreath product algebra, for which we emulate results of Chuang and Tan. Finally, we make extensive use of known decomposition numbers, which have been worked out by Maas \cite{Ma}.

It is natural to ask what the corresponding results are for $p\gs5$, and it seems likely that many of our techniques will help in finding a general classification. With this in mind, we prove our results for arbitrary odd $p$ as far as possible until \cref{secdimarg} (especially results on induction and restriction in \cref{indressec}). However, in \cref{wreathsec,proofmainthm} we will restrict attention to characteristic $3$. The main difficulty in generalising \cref{mainthm} seems to be in generalising the list of ``exceptional'' partitions appearing in (\ref{smallcases}), which appears to grow as $p$ gets larger. But with limited data available for small values of $n$, it is hard to make a conjecture.

\subsection*{Acknowledgements}

The first author was supported during this research by EPSRC Small Grant EP/W005751/1. This funding also allowed the second author to visit Queen Mary University of London, where some of this research wad carried out.

While working on the revised version the second author was working at Mathematisches Institut of the Heinrich-Heine-Universität D\"usseldorf as well as the Department of Mathematics of the University of York. While working at the University of York the second author was supported by the Royal Society grant URF$\backslash$R$\backslash$221047.

We thank Sasha Kleshchev and the referee for helpful comments and conversations.

\section{Background}

In this section we describe the background results we shall need. Throughout this paper we fix an odd prime $p$, and we will state results for general $p$ as far as is convenient. Later in the paper we will specialise to the case $p=3$.

\subsection{Partitions}


A \emph{composition} of $n$ is a sequence $\la=(\la_1,\la_2,\dots)$ of non-negative integers such that the sum $|\la|=\la_1+\la_2+\cdots$ equals $n$. A composition $\la$ is called a \emph{partition} if it is weakly decreasing. When writing partitions, we usually group together equal parts with a superscript and omit trailing zeroes. The unique partition of $0$ is written as $\varnothing$. The \emph{length} of a partition $\la$, denoted $l(\la)$, is the number of non-zero parts of $\la$. We write $\Par$ for the set of all partitions, and $\Par(n)$ for the set of partitions of $n$.

We adopt natural conventions for adding together compositions and multiplying by scalars. Suppose $\la,\mu$ are compositions and $m\in\bbn$. Then we write $\la+m\mu$ for the composition $(\la_1+m\mu_1,\la_2+m\mu_2,\dots)$. We also define $\la\sqcup\mu$ to be the partition whose parts are the combined parts of $\la$ and $\mu$, written in decreasing order.

Often we will consider partitions some of whose parts form an arithmetic progression with common difference $3$; so if $x\gs y$ with $x\equiv y\Md3$, then we write $x\df y$ to mean the sequence $x,x-3,\dots,y$. For example, we may write $(18,\ 17\df5,\ 1)$ for the partition $(18,17,14,11,8,5,1)$.

The \emph{dominance order} $\domby$ is a partial order on partitions defined by setting $\la\domby\mu$ if $|\la|=|\mu|$ and $\la_1+\dots+\la_r\ls\mu_1+\dots+\mu_r$ for all $r$.

A partition $\la$ is called \emph{strict} if $\la_r>\la_{r+1}$ for all $1\ls r<l(\la)$, or \emph{$p$-strict} if for all $r$ either $\la_r>\la_{r+1}$ or $\la_r\equiv0\Md p$. A $p$-strict partition is called \emph{restricted} if for all $r$ either $\la_r-\la_{r+1}<p$ or $\la_r-\la_{r+1}=p$ and $\la_r\nequiv0\Md p$. We write $\RP$ for the set of all strict partitions, and $\RP(n)$ for the set of strict partitions of $n$. We write $\rpp$ for the set of all restricted $p$-strict partitions, and $\rpp(n)$ for the set of restricted $p$-strict partitions of $n$. We say that a strict partition is \emph{even} is it has an even number of positive even parts, and \emph{odd} otherwise.

For any partition $\la$, we write $l_p(\la)$ for the number of positive parts of $\la$ divisible by $p$.

\subsection{Addable and removable nodes}

The Young diagram of a partition $\la$ is the set
\[
\lset{(r,c)\in\Z^2}{1\ls r,\ 1\ls c\ls \la_r},
\]
whose elements we call the \emph{nodes} of $\la$. We draw Young diagrams as arrays of boxes in the plane using the English convention.

The \emph{residue} of a node $(r,c)$ is the smaller of the residues of $c-1$ and $r-c$ modulo $p$; an \emph{$i$-node} means a node of residue $i$. We write $I=\{0,\dots,\frac12(p-1)\}$ for the set of possible residues.

If $\la$ is a $p$-strict partition, then we say that $\la$ is \emph{$p$-even} if it has an even number of nodes of non-zero residue, and \emph{$p$-odd} otherwise.

For example, suppose $p=5$ and $\la=(8,7,3)$. The residues of the nodes of $\la$ are indicated in the following diagram.
\[
\young(01210012,0121001,012)
\]
Since $\la$ has eleven nodes of non-zero residue, it is a $5$-odd partition.

If $\la$ is a $p$-strict partition, then we say that an $i$-node of $\la$ is \emph{\rmbl} if it can be removed (possibly in conjunction with other $i$-nodes) to leave a $p$-strict partition. We define \adbl nodes in a similar way.

If $\la$ is a strict partition, we say that an $i$-node of $\la$ is \emph{\srmbl} if it can be removed (possibly in conjunction with other $i$-nodes) to leave a smaller strict partition. We define \sadbl $i$-nodes of $\la$ similarly.

(In fact if $i\neq0$ the notions of \adbl and \sadbl coincide when $\la$ is strict, but for $i=0$ they are different, and the distinction will be crucial in this paper.)

The notion of \adbl and \rmbl nodes leads to the definition of normal and conormal nodes. Take $\la\in\rpp$, and let $i$ be a residue. The \emph{$i$-signature} of $\la$ is the sequence of signs obtained by reading the \adbl and \rmbl $i$-nodes of $\la$ from left to right, writing $+$ for each \adbl node and $-$ for each \rmbl node. The \emph{reduced $i$-signature} is obtained by repeatedly deleting adjacent pairs $+-$. The \rmbl nodes corresponding to the $-$ signs in the reduced $i$-signature are called the \emph{normal} $i$-nodes of $\la$, and the \adbl nodes corresponding to the $+$ signs are called the \emph{conormal} $i$-nodes of~$\la$.

If $\la$ has at least one normal $i$-node, then we define $\te_i(\la)$ to be the partition obtained by removing the rightmost normal $i$-node. If $\la$ has at least one conormal $i$-node, then we define $\tf_i(\la)$ to be the partition obtained by adding the leftmost conormal $i$-node. It is a simple exercise to check that $\te_i(\la)$ and $\tf_i(\la)$ are restricted $p$-strict partitions if they are defined, and that if $\la,\mu$ are restricted $p$-strict partitions, then $\la=\te_i(\mu)$ \iff $\mu=\tf_i(\la)$.

For example, take $p=3$, $\la=(5,4,3,2,1)$ and $i=0$. The \adbl and \rmbl $0$-nodes are indicated in the following diagram.
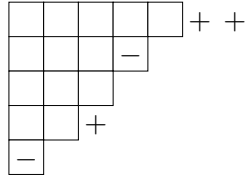
\[
\gyoung(;;;;;:+:+,;;;;-,;;;,;;:+,-)
\]
We see that the $0$-signature of $\la$ is $-+-++$. So the reduced $0$-signature is $-++$; the normal $0$-node is $(5,1)$, and the conormal $0$-nodes are $(1,6)$ and $(1,7)$. So $\te_0(\la)=(5,4,3,2)$ and $\tf_0(\la)=(6,4,3,2,1)$.

\subsection{Regularisation}\label{regnsec}

For each $l\gs0$, we define the $l$th \emph{ladder} to be the set of nodes
\[
\ladd l=\rset{(r,c)\in\bbn^2}{\left\lfloor\mfrac{(p-1)c}p\right\rfloor+(p-1)(r-1)=l}.
\]
For example when $p=3$, the ladders can be illustrated in the following diagram, where we label all the nodes in $\ladd l$ with $l$, for each $l$.
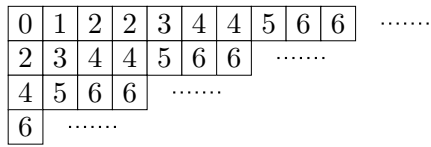
\[
\gyoung(0122344566^\hf^2\hdts,2344566^\hf^2\hdts,4566^\hf^2\hdts,6^\hf^2\hdts)
\]
We say that ladder $m$ is \emph{longer} than ladder $l$ when $l<m$.

The main reason for introducing ladders is to define \emph{regularisation}. Given a $p$-strict partition $\la$, we define its regularisation $\la^\reg$ by taking the Young diagram of $\la$ and moving all the nodes to the leftmost positions in their ladders. It is a simple exercise to show that the resulting diagram is the Young diagram of a restricted $p$-strict partition.

For example, if $p=3$ and $\la=(12,7,2)$, then $\la^\reg=(8,6,4,2,1)$, as we see from the following diagrams.
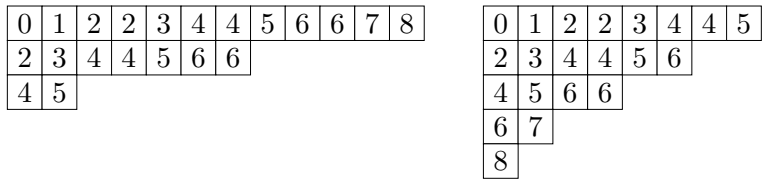
\[
\young(012234456678,2344566,45)\qquad\young(01223445,234456,4566,67,8)
\]

\subsection{$p$-bar cores}

Next we recall the notion of $p$-bar cores. Suppose $\la$ is a $p$-strict partition. \emph{Removing a $p$-bar} from $\la$ means either
\begin{itemize}
\item
replacing $\la_r$ with $\la_r-p$ and then re-ordering, for some $r$ for which $\la_r\gs p$ and either $p\mid\la_r$ or $\la_r-p$ is not a part of $\la$, or
\item
deleting two parts whose sum is $p$.
\end{itemize}

The $p$-bar core of $\la$ is the strict partition obtained by repeatedly removing $p$-bars until none remain. The $p$-bar-weight of $\la$ is the number of $p$-bars removed to reach the $p$-bar core. In general, we say that a strict partition is a $p$-bar core if it equals its own $p$-bar core.

For example, the $3$-bar cores are the partitions $(3l-1\df2)$ for $l\gs0$ and $(3l-2\df1)$ for $l\gs1$.

\subsection{The double cover of the symmetric group}\label{doublecover}

Now we set out the background on representation theory that we need. We work over a sufficiently large field $\F$ of characteristic not equal to $2$. 
Let $\sym_n$ denote the symmetric group on $\{1,\dots,n\}$, and $\hsym_n$ the double cover of $\sym_n$ specified in \cite{KBook}. The group algebra $\F\hsym_n$ decomposes as a direct sum of $\F\sym_n$ and a twisted group algebra $\calt_n$. For $n\gs4$ the group $\hsym_n$ is a Schur cover of $\sym_n$, and representations of $\calt_n$ are called \emph{spin} representations of $\hsym_n$.

In practice it is easier to treat $\calt_n$ as a superalgebra (that is, a $\bbz/2\bbz$-graded algebra), and consider its supermodules. Much of the representation theory of $\calt_n$ as an algebra can then be recovered by forgetting the $\bbz/2\bbz$-grading. We refer to \cite[Chapter 12]{KBook} for the essentials on superalgebras and supermodules. In particular, we note that irreducible supermodules come in two types: an irreducible supermodule $D$ is of \emph{type M} if it is irreducible as a module, and of \emph{type Q} if it is reducible as a module (in which case it is a direct sum of two irreducible modules). In any block of $\calt_n$ all the irreducible modules have the same type, and we say that the block is of type M or type Q accordingly.

Given modules or supermodules $M,N$, we write $M\sim N$ to mean that $M$ and $N$ have the same composition factors. We may extend this notation linearly, writing $M\sim N+N'$ or $M\sim aN$ for $a\gs0$.

The classification of irreducible supermodules in characteristic $0$ can be derived from the work of Schur. For each $\la\in\RP(n)$ there is an irreducible supermodule $\s^\la_0$, of type M if $\la$ is even, and of type Q if $\la$ is odd. These supermodules give all the irreducible $\calt_n$-supermodules up to isomorphism when $\bbf$ has characteristic $0$.

The classification of irreducible supermodules in odd characteristic $p$ is much more recent, and is due to Brundan and Kleshchev \cite{BK4,BK2}. (The constructions in the two papers are different but they have been shown to be equivalent in \cite{ks}.) For each $\la\in\rpp(n)$ there is an irreducible supermodule $\D^\la$, of type M if $\la$ is $p$-even, or of type Q if $\la$ is $p$-odd. We write $\P^\la$ for the projective cover of~$\D^\la$.

For $p>0$, we will write $\s^\la_p$ (or simply $\s^\la$ if $p$ is understood) for a $p$-modular reduction of $\s^\la_0$. Though $\s^\la$ is not well-defined as a supermodule, its composition factors are. If $P$ is a projective supermodule for $\hsym_n$ in characteristic $p$, then $P$ lifts to a $\bbc\hsym_n$-module $P_0$; we write $[P:\s^\la]$ for the composition multiplicity of $\s^\la_0$ in $P_0$.
    
The \emph{decomposition number problem} asks for the composition multiplicities $d_{\la\mu}=[\s^\la:\D^\mu]$ for $\la\in\RP(n)$ and $\mu\in\rpp(n)$. By Brauer reciprocity $d_{\la\mu}$ can also be recovered from the multiplicity $[\P^\mu:\s^\la]$; specifically, $d_{\la\mu}=2^{x-y}[\P^\mu:\s^\la]$, where $x=1$ if $\la$ is odd and $0$ otherwise, while $y=1$ if $\mu$ is $p$-odd and $0$ otherwise.

We will also need the classification of the blocks of $\calt_n$, which was proved by Humphreys \cite{humph}. In fact in this paper we consider superblocks; since these are almost always the same as blocks, we will abuse notation by using the term ``block'' to mean ``superblock''. Given a strict partition $\la$ and a restricted $p$-strict partition $\mu$, the supermodules $\s^\la$ and $\D^\mu$ lie in the same block of $\calt_n$ \iff $\la$ and $\mu$ have the same $p$-bar core. This automatically means that $\la$ and $\mu$ have the same $p$-bar-weight, and it customary to label a block by its bar core and bar-weight, meaning the common $p$-bar core and $p$-bar-weight of the partitions labelling modules in the block.

An alternative description of the blocks follows from a combinatorial lemma of Morris and Yaseen \cite[Theorem 5]{my}: $\s^\la$ and $\D^\mu$ lie the same block \iff $\la$ and $\mu$ have the same number of $i$-nodes, for each $i\in I$. So we may alternatively label a block by the multiset of elements of $I$ corresponding to the residues of a labelling partition. We call this multiset the \emph{content} of the block.

\subsection{Induction and restriction functors}\label{indresfunc}

We shall make extensive use of results concerning induction and restriction functors; these are due to Brundan and Kleshchev, and are explained in more detail in the survey \cite{bkdurham} and the book \cite{KBook}. From now on wee assume $\bbf$ has odd characteristic $p$.

Recall from above the notion of the \emph{content} of a block of $\calt_n$. Suppose $M$ is a $\calt_n$-module lying in a block $B$ and $i\in I$. Let $B(-i)$ be the block of $\calt_{n-1}$ whose content is obtained from the content of $B$ by removing a copy of $i$ (if such a block exists), and let $\Res_iM$ be the block component of $M\da_{\calt_n}$ lying in $B(-i)$ (if there is no block $B(-i)$, then set $\Res_iM=0$). We define induction functors $\Ind_i$ for $i\in I$ in a similar way by picking out block components of induced modules. Then
\[
M\da_{\calt_{n-1}}=\bigoplus_{i\in I}\Res_iM,\qquad M\ua^{\calt_{n+1}}=\bigoplus_{i\in I}\Ind_iM.
\]
The functors $\Res_i,\Ind_i$ are defined for all $n$, so we can consider powers $\Res_i^r,\Ind_i^r$, for $r\gs0$. For any module $M$ and any $i\in I$ there is $r\gs0$ for which $\Res_i^rM=0$, and an important part of the Brundan--Kleshchev branching rules is a determination of the minimal such $r$ in the case where $M=\s^\la$ or~$\D^\mu$.

Take $\la\in\RP(n)$ and $i\in I$. Write $\la^{-i}$ for the partition obtained by removing all the \srmbl $i$-nodes of $\la$, and let $\heps_i(\la)$ be the number of nodes removed; that is, $\heps_i(\la)=|\la|-|\la^{-i}|$. Similarly, write $\la^{+i}$ for the partition obtained by adding all the \sadbl $i$-nodes to $\la$, and let $\hphi_i(\la)$ be the number of $i$-nodes added to $\la$ to obtain $\la^{+i}$.

For example, suppose $p=3$ and $\la=(9,5,4,2)$. The \sadbl and \srmbl $0$-nodes are indicated in the following diagram.
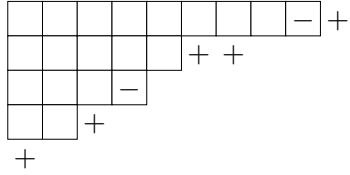
\[
\gyoung(;;;;;;;;;-:+,;;;;;:+:+,;;;;-,;;:+,:+)
\]
We see that $\la^{-0}=(8,5,3,2)$ and $\la^{+0}=(10,7,4,3,1)$, so that $\heps_0(\la)=2$ and $\hphi_0(\la)=5$.

The following is a version of the classical branching rule for spin modules in characteristic $0$.

\begin{theor}\label{branchingrule}
Suppose $\la\in\RP(n)$ and $i\in I$.
\begin{enumerate}
\item
Let $\La^-$ be the set of strict partitions that can be obtained by removing an $i$-node from $\la$. Then
\[
\Res_i\s^\la\sim\sum_{\mu\in\La^-}a_\mu\s^\mu,
\]
where $a_\mu$ equals $2$ if $\la$ is odd and $\mu$ is even, and $1$ otherwise.
\item
Let $\La^+$ the set of strict partitions that can be obtained by adding an $i$-node to $\la$. Then
\[
\Ind_i\s^\la\sim\sum_{\mu\in\La^+}a_\mu\s^\mu,
\]
where $a_\mu$ equals $2$ if $\la$ is odd and $\mu$ is even, and $1$ otherwise.
\end{enumerate}
\end{theor}

\begin{cor}\label{branchs}
Suppose $\la\in\RP(n)$ and $i\in I$, and let $\heps=\heps_i(\la)$ and $\hphi=\hphi_i(\la)$.
\begin{enumerate}
\item
There is $a>0$ such that
\[
\Res_i^\heps\s^\la\sim a\s^{\la^{-i}},\qquad\Res_i^r\s^\la=0\text{ for }r>\heps.
\]
\item
There is $b>0$ such that
\[
\Ind_i^\hphi\s^\la\sim b\s^{\la^{+i}},\qquad\Ind_i^r\s^\la=0\text{ for }r>\hphi.
\]
\end{enumerate}
\end{cor}

For the irreducible modules $\D^\mu$ we can make a similar statement to \cref{branchs}. Given a restricted $p$-strict partition $\mu$, let $\eps_i(\mu)$ be the number of normal $i$-nodes of $\mu$, and let $\mu\im$ be the partition obtained by removing all the normal $i$-nodes from $\mu$ (or in other words, $\mu\im=\te_i^{\eps_i(\mu)}(\mu)$). Similarly, let $\phi_i(\mu)$ be the number of conormal $i$-nodes of $\mu$, and let $\mu\ip=\tf_i^{\phi_i(\mu)}(\mu)$ be the partition obtained by adding all the conormal $i$-nodes.

\begin{lemma}\label{branchd}
Suppose $\mu\in\rpp(n)$ and $i\in I$, and let $\eps=\eps_i(\mu)$ and $\phi=\phi_i(\mu)$.
\begin{enumerate}
\item
There is $a>0$ such that
\[
\Res_i^\eps\D^\mu\sim a\D^{\mu\im},\qquad\Res_i^r\D^\mu=0\text{ for }r>\eps.
\]
\item
There is $b>0$ such that
\[
\Ind_i^\phi\D^\mu\sim b\D^{\mu\ip},\qquad\Ind_i^r\D^\mu=0\text{ for }r>\phi.
\]
\end{enumerate}
\end{lemma}

\subsection{Jucys--Murphy elements and weight spaces}

Later we shall need to make use of weight spaces; we briefly explain the background, following \cite[Section 22]{KBook}. 

There are elements $y_1,\dots,y_{n}\in\calt_n$ (defined originally by Sergeev) which are analogues of the Jucys--Murphy elements in $\F\sym_n$; in particular, the squares $y_1^2,\dots,y_{n}^2$ generate a large commutative subalgebra of $\calt_n$. Given a $\calt_n$-supermodule $M$ and a tuple $i=(i_1,\dots,i_n)\in I^n$, we define the \emph{$i$-weight space}
\[
M[i]=\lset{m\in M}{(y_r^2-\tfrac12i_r(i_r+1))^Nm=0\text{ for $N\gg0$ and $1\ls r\ls n$}}.
\]
Then $M$ is the direct sum of the weight spaces $M[i]$. We say that $i$ is a \emph{weight} of $M$ if $M[i]\neq0$.

In the case where $M=\s^\la$ for $\la\in\RP(n)$ we can describe the weights explicitly. To explain this, we need to introduce \emph{standard shifted tableaux}. A standard shifted $\la$-tableau is a bijection $t$ from the Young diagram of $\la$ to the set $\{1,\dots,n\}$ such that
\[
t(r,c)<t(r,c+1),\qquad t(r,c+1)<t(r+1,c)
\]
for all admissible $r,c$.

Given a shifted tableau $t$, we define $i_k$ to be the residue of the node $t^{-1}(k)$ for each $k$, and then define the tuple $i^t=(i_1,\dots,i_n)\in I^n$. Now the following result follows from the branching rules for the modules $\s^\la$.

\begin{propn}\label{tabweight}
Suppose $\la\in\RP(n)$ and $t$ is a standard shifted $\la$-tableau. Then $i^t$ is a weight of~$\s^\la$.
\end{propn}

\begin{proof}
We use induction on $n$, with the case $n=0$ being trivial. Assuming $n\gs1$, let $(r,c)$ be the node for which $t(r,c)=n$. Let $\la^-$ be the partition obtained from $\la$ by removing the node $(r,c)$. Then the definition of a standard shifted tableau means that $\la^-\in\RP(n-1)$ and that the tableau $t^-$ obtained by restricting $t$ to $\la^-$ is a standard shifted $\la^-$-tableau. Let $i$ be the residue of $(r,c)$. Then the branching rule means that $\s^{\la^-}$ appears in $\Res_i\s^\la$, so by \cite[Lemma 22.3.14]{KBook} if $(i_1,\dots,i_{n-1})$ is a weight of $\s^{\la^-}$ then $(i_1,\dots,i_{n-1},i)$ is a weight of $\s^\la$. By induction $i^{t-}$ is a weight of $\s^{\la^-}$, and so $i^t$ is a weight of~$\s^\la$.
\end{proof}

\section{Irreducible and homogeneous modules}

Our motivation in this paper is to classify the irreducible representations of $\hsym_n$ in characteristic $0$ which remain irreducible when reduced modulo $p$. In fact the answer to this question depends on whether we consider modules or supermodules, but we address both questions (as well as the corresponding question for the double cover of the alternating group) by considering the more general question of homogeneous modules.

We say that a strict partition $\la$ is \emph{homogeneous} if the composition factors of $\s^\la$ (as a supermodule) are all isomorphic. The following \lcnamecref{L300921} due to Brundan and Kleshchev is a very useful tool for classifying homogeneous partitions. Recall that $\la^\reg$ denotes the regularisation of $\la$, and $l_p(\la)$ the number of positive parts of $\la$ divisible by $p$.

\begin{theor}\label{L300921}
Suppose $\la\in\RP$. Then
\begin{align*}
[\s^\la:\D^{\la^\reg}]&=2^{\frac{l_p(\la)+x-y}{2}},
\\
[\P^{\la^\reg}:\s^\la]&=2^{\frac{l_p(\la)+y-x}{2}},
\end{align*}
where $x$ is $0$ if $\la$ is even and $1$ if $\la$ is odd, and $y$ is $0$ if $\la$ is $p$-even and $1$ if $\la$ is $p$-odd.  All other composition factors of $\s^\la$ are of the form $\D^\mu$ with $\mu\lhd\la^\reg$.
\end{theor}

\begin{proof}
The corresponding result for the algebraic supergroup $Q(n)$ is proved in \cite[Theorem 1.2]{BK3}, and the result for $\calt_n$ follows by \cite[Theorem 10.8]{BK2}.
\end{proof}

Suppose $\la$ is a strict partition. If $\la$ is even, then $\s^\la_0$ is irreducible as a module, and we write this module as $\ss^\la_0$. If $\la$ is odd, then as a module $\s^\la_0$ splits as the direct sum of two non-isomorphic irreducible modules which we write as $\ss^{\la,+}_0$ and $\ss^{\la,-}_0$. These modules give all the irreducible spin representations of $\hsym_n$ in characteristic $0$. We also consider the double cover $\halt_n$ of the alternating group: if $\la$ is even, then $\ss^\la_0\dalt$ splits as the direct sum of two non-isomorphic irreducible modules which we write as $\st^{\la,+}_0$ and $\st^{\la,-}_0$, while if $\la$ is odd, then $\ss^{\la,+}_0\dalt\cong\ss^{\la,-}_0\dalt$ is an irreducible module which we write as $\st^\la_0$. These modules give all the irreducible spin representations of $\halt_n$ in characteristic $0$. We write $\ss^\la_p$ for a $p$-modular reduction of $\ss^\la_0$, and similarly for the other modules just defined (though we may simply write $\ss^\la$ if $p$ is understood).

A classification of homogeneous partitions also directly allows us to classify which modules $S^{\la(,\pm)}_p$ and $T^{\la(,\pm)}_p$ are \emph{almost homogeneous} (by which we mean that all composition factors of $S_p^{\la(,\pm)}$ or $T_p^{\la(,\pm)}$ are indexed by the same partition), as $S_p^{\la(,\pm)}$ has a composition factor indexed by $\mu$ if and only if $\s_p^\la$ has a composition factor $\D_p^\mu$ and similarly for $T_p^{\la(,\pm)}$.

\cref{L300921} tells us that if $\la$ is homogeneous then all the composition factors of $\s_p^\la$ are isomorphic to $\D^{\la^\reg}$, and gives the number of composition factors. So a classification of homogeneous partitions also allows us to answer the question of irreducibility in different settings, as follows.

\begin{propn}\label{irredhomog}
Suppose $\la\in\RP$.
\begin{enumerate}
\item
\begin{enumerate}
\item
If $\la$ is even, then $\s^\la_p$ is an irreducible supermodule \iff $\la$ is homogeneous and $l_p(\la)\ls1$.
\item
If $\la$ is odd, then $\s^\la_p$ is an irreducible supermodule \iff $\la$ is homogeneous and $l_p(\la)=0$.
\end{enumerate}
\item\label{symirred}
\begin{enumerate}[ref=\alph*]
\item\label{symirredeven}
If $\la$ is even, then $\ss^\la_p$ is irreducible \iff $\la$ is homogeneous and $l_p(\la)=0$.
\item\label{symirredodd}
If $\la$ is odd, then $\ss^{\la,\pm}_p$ is irreducible \iff $\la$ is homogeneous and $l_p(\la)\ls1$.
\end{enumerate}
\item\label{altirred}
\begin{enumerate}
\item
If $\la$ is even, then $\st^{\la,\pm}_p$ is irreducible \iff $\la$ is homogeneous and $l_p(\la)\ls1$.
\item
If $\la$ is odd, then $\st^\la_p$ is irreducible \iff $\la$ is homogeneous and $l_p(\la)=0$.
\end{enumerate}
\end{enumerate}
\end{propn}

\begin{proof}
In this proof, for any supermodule $M$ we write $\underline M$ for the underlying module.
\begin{enumerate}
\item
This follows immediately from \cref{L300921}.
\item
\begin{enumerate}
\item
In this case $\ss^\la=\underline{\s^\la}$, so if $\ss^\la$ is irreducible then certainly $\s^\la$ is. So assume $\la$ is homogeneous and $l_p(\la)\ls1$. Then $\s^\la\cong\D^{\la^\reg}$, by \cref{L300921}. Now $\ss^\la\sim\underline{\D^{\la^\reg}}$, which is irreducible \iff $\la^\reg$ is $p$-even. By \cref{L300921}, this happens \iff $l_p(\la)=0$.
\item
First note that $\ss^{\la,+}$ is irreducible \iff $\ss^{\la,-}$ is, since these modules are images of each other under tensoring with the sign module. Now $\ss^{\la,+}\oplus\ss^{\la,-}\sim\underline{\s^\la}$, so if $\ss^{\la,\pm}$ is irreducible then certainly $[\s^\la:\D^{\la^\reg}]\ls2$.

If $[\s^\la:\D^{\la^\reg}]=1$, then by \cref{L300921} $l_p(\la)=0$ and $\la$ is $p$-odd, so $\underline{\D^{\la^\reg}}$ has two composition factors. So $\ss^{\la,\pm}$ is irreducible \iff $\s^\la$ has no other composition factors, i.e.\ $\la$ is homogeneous.

If $[\s^\la:\D^{\la^\reg}]=2$, then by \cref{L300921} either $l_p(\la)=1$ and $\la$ is $p$-even, or $l_p(\la)=2$ and $\la$ is $p$-odd. But in the second case $\underline{\s^\la}$ would have at least four composition factors, a contradiction. So $\ss^{\la,\pm}$ can be irreducible only in the first case, with $\la$ being homogeneous.
\end{enumerate}
\item
We begin by observing that if $\mu$ is a restricted $p$-strict partition then $\underline{\D^\mu}\dalt$ has composition length~$2$: if $\mu$ is $p$-even then $\underline{\D^\mu}$ is irreducible and therefore invariant under tensoring with the sign module, so its restriction to $\halt_n$ splits as a direct sum of two modules. If $\mu$ is $p$-odd, then $\underline{\D^\mu}$ is a direct sum of two non-isomorphic modules; these modules are images of each other under tensoring with the sign module, so remain irreducible on restriction to $\halt_n$.
\begin{enumerate}
\item
As in \ref{symirred}(\ref{symirredodd}), $\st^{\la,+}$ is irreducible \iff $\st^{\la,-}$ is. Now $\st^{\la,+}\oplus\st^{\la,-}\sim\underline{\s^\la}\dalt$.  So (from the observation above) $\st^{\la,\pm}$ is irreducible \iff $\s^\la$ is an irreducible supermodule, i.e.\ \iff $l_p(\la)\ls1$.
\item
In this case $\underline{\s^\la}\dalt\sim\ss^{\la,+}\dalt\oplus\ss^{\la,-}\dalt\sim2\st^\la$. So $\st^\la$ is irreducible \iff $\underline{\s^\la}\dalt$ has exactly two composition factors, which happens \iff $\s^\la$ is irreducible.\qedhere
\end{enumerate}
\end{enumerate}
\end{proof}

With \cref{irredhomog} in mind, we focus on homogeneous modules for the rest of the paper.

To begin with, we will need some simple known results on decomposition numbers due to Wales and M\"uller. For simplicity we state M\"uller's result only in the case $p=3$.

\begin{theor}\cite[Tables III and IV]{Wales}\label{walthm}
\begin{enumerate}
\item
Suppose $\la=(n)$. Then $\la$ is homogeneous.
\item
Suppose $\la=(n-1,1)$ with $n\gs5$. Then $\la$ is homogeneous \iff $n\nequiv0,1\ppmod p$.
\end{enumerate}
\end{theor}

\begin{theor}\cite[Theorem 4.4]{Mu}\label{mullerthm}
Suppose $p=3$, and that $\la=\nu\sqcup(3)$ with $\nu$ a $3$-bar core. Then $\la$ is homogeneous.
\end{theor}

We now reformulate our main result \cref{mainthm}, giving us a classification of homogeneous $\la$ excluding partitions of a certain type. We say that a strict partition $\la$ is \emph{\spc} if there is $i\in\{1,2\}$ such that all the non-zero parts of $\la$ are congruent to $i$ modulo $3$. Equivalently, $\la$ has the form $\nu+3\al$ where $\nu$ is a $3$-bar core and $\al$ is a partition with $l(\al)\ls l(\nu)$.

\begin{theor}\label{T081021}
Suppose $p=3$ and $\la$ is a strict partition which is not \spc. Then $\la$ is homogeneous \iff one of the following holds.
\begin{enumerate}
\item
$\la=(3a)$ with $a\geq 2$.
\item
$\la=\nu\sqcup(3)$ with $\nu$ a $3$-bar core.
\item
$\la\in\{(2,1),(3,2,1),(4,3,2),(4,3,2,1),(5,3,2,1),(5,4,3,1),$ 

$(5,4,3,2),(5,4,3,2,1),(7,4,3,2,1),(8,5,3,2,1)\}$.
\end{enumerate}
\end{theor}

We can immediately observe the following consequence, which is relevant for \cref{irredhomog}.

\begin{cor}
Suppose $p=3$ and $\la\in\RP$. If $\la$ is homogeneous, then $l_3(\la)\ls1$.
\end{cor}

\begin{proof}
We can easily see that each of the partitions in \cref{T081021} has at most one part divisible by $3$. Any other homogeneous partition is \spc, and so by definition has no parts divisible by~$3$.
\end{proof}

For \spc partitions, we make the following conjecture. A partition $\al$ is called \emph{$3$-Carter} if for every $1\ls r<s$ and every $1\ls c\ls\al_s$, the $(r,c)$-hook length of $\al$ and the $(s,c)$-hook length of $\al$ are divisible by the same powers of $3$. (The name ``Carter'' refers to a conjecture formulated by Carter, and proved by James \cite{jamescarter}, on irreducible Specht modules for the symmetric group.)

\begin{conj}\label{spcconj}
Suppose $p=3$ and $\la$ is a \spc strict partition, and write $\la=\nu+3\alpha$ with $\nu$ a $3$-bar core. Then $\la$ is homogeneous \iff $\al$ is a $3$-Carter partition.
\end{conj}

We expect this conjecture to hold in view of \cite[Conjecture 1]{KL} and \cite[Theorem 6.17]{FKM}. The results in this paper show that it is true whenever $l(\al)=1$, or the last column of $\al$ has length at least $3$, or the last two columns of $\al$ both have length $2$. Known decomposition matrices also show that the conjecture is true when $\al=(2,1)$ or $(3,1)$.

We will give the proof of \cref{T081021} in the final section of this paper. Before that, we recall some other results we shall need and prove some useful results.

\section{Induction and restriction of homogeneous modules}\label{indressec}

In this section we apply the induction and restriction functors from \cref{indresfunc} to homogeneous modules, in order to provide a basis for the induction proof of our main theorem.

\begin{lemma}\label{T300921}
Let $\la\in\RP(n)$ and $i\in I$, and assume that $\la$ is homogeneous. Then $\la^{-i}$ is homogeneous. Furthermore $\heps_i(\la)=\eps_i(\la^\reg)$ and $(\la^\reg)\im=(\la^{-i})^\reg$.
\end{lemma}

\begin{proof}
By \cref{L300921} we have $\s^\la\sim(\D^{\la^\reg})^{\oplus c}$ for some $c>0$. Since the functor $\Res_i$ is exact \cref{branchs,branchd} immediately give $\heps(\la)=\eps(\la^\reg)$ and
\[
(\s^{\la^{-i}})^{\oplus a}\sim(\D^{(\la^\reg)\im})^{\oplus b}
\]
for some $a,b>0$. Now \cref{L300921} gives $(\la^{-i})^\reg=(\la^\reg)\im$.
\end{proof}

In the same way we can prove the following.

\begin{lemma}\label{T300921_2}
Let $\la\in\RP(n)$ and $i\in I$, and assume that $\la$ is homogeneous. Then $\la^{+i}$ is homogeneous. Furthermore $\hphi_i(\la)=\phi_i(\la^\reg)$ and $(\la^\reg)\ip=(\la^{+i})^\reg$.
\end{lemma}

\begin{cor}\label{C081021}
Let $\la\in\RP(n)$ and suppose $i\in I$ with $\hphi_i(\la)=0$. Then $\la$ is homogeneous if and only if $\la^{-i}$ is homogeneous.
\end{cor}

\begin{proof}
This follows by \cref{T300921,T300921_2}.
\end{proof}

Now we prove an analogue of a result from~\cite{F} which helps us to apply the condition $(\la^\reg)\im=(\la^{-i})^\reg$ from \cref{T300921}. To help with future work, we continue to allow $p$ to be any odd prime, though our results simplify considerably in the special case $p=3$ that we consider in this paper.

We recall the definition of the ladders $\ladd l$ from \cref{regnsec}. Observe that the nodes in $\ladd l$ all have the same residue, and that this residue depends on the residue of $l$ modulo $p-1$.

Suppose $\la$ is a $p$-strict partition, and let $\lad l\la$ denote the number of nodes of $\la$ in $\ladd l$. Let $\add l\la$ denote the number of \sadbl nodes of $\la$ in $\ladd l$, and $\badd l\la$ the number of \adbl nodes of $\la$ in $\ladd l$. Define $\rem l\la$ and $\brem l\la$ similarly, for \srmbl nodes. Define all these numbers to be $0$ when $l<0$.

We begin by considering ladders comprising nodes of residue $\frac12(p-1)$.

\begin{lemma}\label{arladd1}
Suppose $\la$ is a $p$-strict partition, and $l\in\bbn$ with $l\equiv\frac12(p-1)\ppmod{p-1}$. Then
\[
\brem{l-p+1}\la-\badd l\la=\begin{cases}
\lad l\la-\lad{l-1}\la-\lad{l-p+2}\la+\lad{l-p+1}\la&\text{if }p\gs5
\\
\lad l\la-\lad{l-1}\la+\lad{l-2}\la&\text{if }p=3.
\end{cases}
\]
\end{lemma}

\begin{proof}
This is proved in a similar way to \cite[Lemma 4.15]{F}. For each node $(r,c)\in\ladd l$, we can consider the set of (up to) four nodes
\[
\{(r-1,c),(r-1,c+1),(r,c-1),(r,c)\}\cap\bbn^2.
\]
By examining the possible intersections of this set with $[\la]$, we can show that the lemma holds when restricted to just these nodes. Then by summing over all nodes $(r,c)\in\ladd l$, we obtain the result.
\end{proof}

When $\la$ is a strict partition and $\ladd l$ consists of nodes of non-zero residue, it is easy to see that $\badd l\la=\add l\la$ and $\brem{l-p+1}\la=\rem{l-p+1}\la$. As a consequence, we obtain
\[
\add l\la-\rem{l-p+1}\la=\badd l{\la^\reg}-\brem{l-p+1}{\la^\reg}
\]
for any strict partition $\la$ when $l\equiv\frac12(p-1)\ppmod{p-1}$, because $\lad k{\la^\reg}=\lad k\la$ for every $k$. Now we can obtain a precise criterion for when $\dn i\la$ in the case $i=\frac12(p-1)$. This is analogous to \cite[Proposition 4.9]{F3}.

\begin{propn}\label{dn1}
Suppose $\la$ is a strict partition and that $i=\frac12(p-1)$. Then $\dn i\la$ \iff $\la$ has a \srmbl $i$-node and a \sadbl $i$-node in a longer ladder.
\end{propn}

\begin{proof}
The proof is essentially the same as for \cite[Proposition 4.9]{F3}. The crucial point is that when we read the \adbl and \rmbl $i$-nodes of a restricted $p$-strict partition from left to right, the nodes in longer ladders come first, and within each ladder the \rmbl nodes come before the \adbl nodes.
\end{proof}

For residues $i\neq\frac12(p-1)$, it is harder to give a necessary and sufficient criterion for $\dn i\la$, but it suffices for our purposes to show that the condition in \cref{dn1} is sufficient.

We consider the case $i=0$. Take $l\equiv0\ppmod{p-1}$, and define $\str l\la$ to be the number of nodes $(r,c)\in\ladd l$ such that $c\equiv0\ppmod3$, $r\gs2$ and $(\la_{r-1},\la_r,\la_{r+1})=(c+1,c,c-1)$.

\begin{lemma}\label{lads}
Suppose $\la$ is a $p$-strict partition and $l\equiv0\ppmod{p-1}$. Then
\begin{align*}
\brem{l-p+1}\la-\badd l\la&=\lad l\la-2\lad{l-1}\la-2\lad{l-p+2}\la+\lad{l-p+1}\la-\delta_{l0}.
\\
\intertext{Furthermore, if $\la$ is strict, then}
\rem{l-p+1}\la-\add l\la&=\lad l\la-2\lad{l-1}\la-2\lad{l-p+2}\la+\lad{l-p+1}\la-\str l\la+\str{l-p+1}\la-\delta_{l0}.
\end{align*}
\end{lemma}

\begin{proof}
Suppose $(r,c)$ is a node in $\ladd l$, with $c\equiv1\ppmod p$. Then (as in \cite[Lemma 4.15]{F}) we can consider the set of nodes
\[
X_{(r,c)}=\{(r-2,c+1),(r-1,c-1),(r-1,c),(r-1,c+1),(r,c-2),(r,c-1),(r,c),(r+1,c-2)\}\cap\bbn^2.
\]
By considering the intersection of $X_{(r,c)}$ with $[\la]$, we can show that if $p\gs5$, then the two formul\ae{} hold when restricted to $X_{(r,c)}$. Now summing over all $(r,c)$ gives the result. (In particular, the calculation for the second formula is identical to that in \cite[Lemma 4.15]{F}.)

For $p=3$, if we restrict attention to $X_{(r,c)}$, we obtain
\begin{align*}
\brem{l-2}\la-\badd l\la&=\lad l\la-2\lad{l-1}\la+\lad{l-2}\la-\delta_{l0}
\\
\intertext{and (if $\la$ is strict)}
\rem{l-2}\la-\add l\la&=\lad l\la-2\lad{l-1}\la+\lad{l-2}\la-\str l\la+\str{l-2}\la-\delta_{l0}.
\end{align*}
Now summing again gives the result, but here we have to take account of the fact that the sets $X_{(r,c)}$ overlap, so that each node in $\ladd{l-1}$ must be counted twice.
\end{proof}

Now we can give a sufficient (but not necessary) condition for $\dn0\la$.

\begin{propn}\label{dn0}
Suppose $\la$ is a strict partition, and that $\la$ has a \srmbl $0$-node and a \sadbl $0$-node in a longer ladder. Then $\dn0\la$.
\end{propn}

\begin{proof}
We follow the proof of \cite[Proposition 4.17]{F}. As there, given an integer $u$ we write $[u]$ to denote a string of $u$ plus signs if $u\gs0$, or a string of $-u$ minus signs if $u\ls0$. Let $\mu=\la^\reg$. When we read the \adbl and \rmbl $0$-nodes from left to right, the nodes in longer ladders come earlier, and within each ladder the \rmbl nodes come before the \adbl nodes. This means that the $0$-signature of $\mu$ is
\[
\dots[\badd{3p-3}\mu][-\brem{2p-2}\mu][\badd{2p-2}\mu][-\brem{p-1}\mu][\badd{p-1}\mu][-\brem0\mu][\badd0\mu].
\]
So the reduced $0$-signature is the reduction of the sequence
\[
\dots[\badd{3p-3}\mu-\brem{2p-2}\mu][\badd{2p-2}\mu-\brem{p-1}\mu][\badd{p-1}\mu-\brem0\mu][\badd0\mu].
\]
By \cref{lads} this is the same as the sequence
\begin{align*}
\dots&[\add{3p-3}\la-\rem{2p-2}\la-\str{3p-3}\la+\str{2p-2}\la]
\\
&[\add{2p-2}\la-\rem{p-1}\la-\str{2p-2}\la+\str{p-1}\la]
\\
&[\add{p-1}\la-\rem0\la-\str{p-1}\la+\str0\la]
\\
&[\add0\la-\str0\la].
\end{align*}
Now the proof proceeds exactly as in \cite[Proposition 4.17]{F}.
\end{proof}

It remains to consider residues $i$ such that $1\ls i\ls\frac12(p-3)$. Take $l\in\bbn$ with $l\nequiv0\ppmod{\frac12(p-1)}$. For a $p$-strict partition $\la$, define $\zz l\la$ to be the number of nodes $(r,c)\in\ladd l$ such that $(\la_r,\la_{r+1})=(c,c-1)$. Now we have the following lemma.

\begin{lemma}\label{zzlem}
Suppose $\la$ is a $p$-strict partition, and $l\in\bbn$ with $l\nequiv0\ppmod{\frac12(p-1)}$. Let $k<l$ be maximal such that $k+l\equiv0\ppmod{p-1}$. Then
\[
\brem k\la-\badd l\la=\begin{cases}
\lad l\la-\lad{l-1}\la-\lad{k+1}\la+\lad k\la-\zz k\la+\zz{l-p+1}\la
&\text{if }l\nequiv1\ppmod{p-1}
\\
\lad l\la-\lad{l-1}\la+\lad k\la-\zz k\la+\zz{l-p+1}\la
&\text{if }l\equiv1\ppmod{p-1}.
\end{cases}
\]
\end{lemma}

\begin{proof}
This is proved in a similar way to \cref{arladd1,lads}. For each node $(r,c)\in\ladd l$, we let $b<c$ be maximal such that $b+c\equiv1\ppmod p$, and consider the set
\[
X_{(r,c)}=\{(r,b),(r,b+1),(r,c-1),(r,c)\}\cap\bbn^2.
\]
It is easy to check that the formula holds when restricted to the set $X_{(r,c)}$, and summing over $(r,c)$ gives the result.
\end{proof}

As with \cref{arladd1}, we can replace $\brem k\la$ and $\badd l\la$ with $\rem k\la$ and $\add l\la$ when $\la$ is strict, because $\ladd k$ and $\ladd l$ consist of nodes of non-zero residue. So we obtain
\[
\add l\la-\rem k\la+\zz{l-p+1}\la-\zz k\la=\badd l{\la^\reg}-\brem k{\la^\reg}+\zz{l-p+1}{\la^\reg}-\zz k{\la^\reg}
\]
when $\la$ is strict. To exploit this, we need to examine the function $\zz l\la$ a little more closely.

\begin{lemma}\label{zzreglem}
Suppose $\la$ is a $p$-strict partition and $l\in\bbn$ with $l\nequiv0\ppmod{\frac12(p-1)}$. Then $\zz l{\la^\reg}\ls\zz l\la$.
\end{lemma}

\begin{proof}
Let $\mu=\la^\reg$. In fact it is possible to determine $\zz l\mu$ quite explicitly. Suppose $\zz l\mu\gs1$, and take a node $(r,c)\in\ladd l$ such that $(\mu_r,\mu_{r+1})=(c,c-1)$. Note in particular that the node $(r,c+1)$ belongs to $\ladd{l+1}\setminus[\mu]$, while $(r+1,c-1)$ belongs to $\ladd{l+p-2}\cap[\mu]$. Now the node $(r+1,c+1-p)$ (if there is such a node) belongs to $[\mu]$, which means that $(r,c+1)$ is the leftmost node in $\ladd{l+1}\setminus[\mu]$. Similarly, $(r+1,c-1)$ is the rightmost node in $\ladd{l+p-2}\cap[\mu]$. This in particular shows that $(r,c)$ is unique, so that $\zz l\mu=1$. We also obtain $\lad{l+p-2}\mu>\lad{l+1}\mu$. Conversely, if $\lad{l+p-2}\mu>\lad{l+1}\mu$ and if there is at least one node in ladder $l+1$ which does not belong to $[\mu]$, then we obtain $\zz l\mu\gs1$ (by taking $(r,c)$ such that $(r,c+1)$ is the leftmost node of $\ladd{l+1}$ not in $\mu$).

So we have shown that $\zz l\mu$ equals $1$ if $\lad{l+p-2}\mu>\lad{l+1}\mu$ and there is at least one node of $\ladd{l+1}$ not in $[\mu]$, and $0$ otherwise. So to prove the \lcnamecref{zzreglem} it suffices to show that if $\lad{l+p-2}\la>\lad{l+1}\la$ and there is at least one node of $\ladd{l+1}$ which is not in $\la$, then $\zz l\la\gs1$. If there is a node $(s,d)\in\ladd{l+p-2}$ such that $s\gs2$ and $(s-1,c+2)\notin[\la]$, then we can deduce $\zz l\la\gs1$ by taking $(r,c)=(s-1,c+1)$, and we are done. So assume there is no such node $(s,d)$. Now the assumption that $\lad{l+p-2}\la>\lad{l+1}\la$ means that $[\la]$ contains the unique node of $\ladd{l+p-2}$ in row $1$, and that for each node $(s,d)\in\ladd{l+p-2}$ with $s\gs2$, the node $(s,d)$ lies in $[\la]$ \iff $(s-1,c+2)$ does. But now we find that because $[\la]$ contains that unique node of $\ladd{l+p-2}$ in row $1$, it also contains the unique node of $\ladd{l+1}$ in row $1$, and therefore contains the unique node of $\ladd{l+p-2}$ in row $2$, and therefore the unique node of $\ladd{l+1}$ in row $2$, and so on, so that $[\la]$ contains every node of $\ladd{l+1}$, contrary to assumption.
\end{proof}

Now we can prove the final result concerning $\dn i\la$.

\begin{propn}\label{dnall}
Suppose $\la$ is a strict partition and $i\in I$, and that $\la$ has a \srmbl $i$-node and a \sadbl $i$-node in a longer ladder. Then $\dn i\la$.
\end{propn}

\begin{proof}
The cases $i=0$ and $i=\frac12(p-1)$ are dealt with in \cref{dn1,dn0}. So take $i\neq0,\frac12(p-1)$. We follow the structure of the proof of \cref{dn0}. Let $\mu=\la^\reg$, and let $l_1<l_2<\cdots$ be the ladders that contain nodes of residue $i$. As in \cref{dn0}, when we read the \adbl and \rmbl $i$-nodes from left to right, the nodes in longer ladders come earlier, and within each ladder the \rmbl nodes come before the \adbl nodes. This means that the $i$-signature of $\mu$ is
\[
\dots[\badd{l_4}\mu][-\brem{l_3}\mu][\badd{l_3}\mu][-\brem{l_2}\mu][\badd{l_2}\mu][-\brem{l_1}\mu][\badd{l_1}\mu].
\]
So the reduced $i$-signature is the reduction of the sequence
\[
\dots[\badd{l_4}\mu-\brem{l_3}\mu][\badd{l_3}\mu-\brem{l_2}\mu][\badd{l_2}\mu-\brem{l_1}\mu][\badd{l_1}\mu].
\]
By \cref{zzlem} this is the same as the sequence
\begin{align*}
\dots&[\add{l_4}\la-\rem{l_3}\la-\dzz{l_3}\la+\dzz{l_2}\la]
\\
&[\add{l_3}\la-\rem{l_2}\la-\dzz{l_2}\la+\dzz{l_1}\la]
\\
&[\add{l_2}\la-\rem{l_1}\la-\dzz{l_1}\la]
\\
&[\add{l_1}\la],
\end{align*}
where $\dzz l\la=\zz l\la-\zz l\mu$ for each $l$. Since by \cref{zzreglem} each $\dzz l\la$ is non-negative, we can complete the proof exactly as in \cite[Proposition 4.17]{F}.
\end{proof}

\section{Dimension arguments}\label{secdimarg}

In the proof of our main theorem we will often use arguments involving dimensions, similar to those in \cite{Fs,F}.

\subsection{The bar-length formula}\label{barlensec}

First we recall Schur's ``bar-length formula'' \cite{schu} for the dimension of $\s^\la$: if $\la\in\RP(n)$, then
\[
\dim\s^\la=2^{\lceil\frac12(n-l(\la))\rceil}\frac{\card\la!}{\prod_{1\ls r\ls l(\la)}\la_r!}\frac{\prod_{1\ls r<s\ls l(\la)}(\la_r-\la_s)}{\prod_{1\ls r<s\ls l(\la)}(\la_r+\la_s)}.
\]
Now define
\[
\ddeg\la=\frac{\dim\s^\la}{[\s^\la:\D^{\la^\reg}]}.
\]
Using \cref{L300921}, we get
\[
\ddeg\la=2^{\lceil\frac12(\card\la-l(\la)-l_p(\la))\rceil}\frac{\card\la!}{\prod_{1\ls r\ls l(\la)}\la_r!}\frac{\prod_{1\ls r<s\ls l(\la)}(\la_r-\la_s)}{\prod_{1\ls r<s\ls l(\la)}(\la_r+\la_s)}.
\]
Now the following lemma follows from \cref{L300921}.

\begin{propn}\label{basicdeg}
Suppose $\la$ and $\mu$ are strict partitions with $\la^\reg=\mu^\reg$ and $\ddeg\la>\ddeg\mu$. Then $\la$ is not homogeneous.
\end{propn}

To help us use this, we have an analogue of \cite[Lemma 4.19]{F}, which is proved in the same way.

\begin{lemma}\label{l419}
Suppose $\la$ and $\mu$ are strict partitions with $\mu\doms\la$, and $m>\mu_1$. Define
\begin{align*}
\la^+&=(m,\la_1,\la_2,\dots),
\\
\mu^+&=(m,\mu_1,\mu_2,\dots).
\end{align*}
Then
\[
\frac{\ddeg{\la^+}}{\ddeg{\mu^+}}>\frac{\ddeg\la}{\ddeg\mu}.
\]
In addition, if $\la^\reg=\mu^\reg$, then $(\la^+)^\reg=(\mu^+)^\reg$.
\end{lemma}

The next few lemmas apply these results in specific situations. \emph{For the remainder of \cref{barlensec}, we take $p=3$.} These results typically work as follows: we define two families of strict partitions $\la_{(l)}$, $\mu_{(l)}$, and we wish to show that $\ddeg{\la_{(l)}}>\ddeg{\mu_{(l)}}$ for all $l$ greater than or equal to some threshold $l_0$. The bar-length formula will allow us to calculate the ratio $r_l=\ddeg{\la_{(l)}}/\ddeg{\mu_{(l)}}$, and to express $r_{l+1}/r_l$ as a rational function of $l$. We show that the value of this rational function is greater than $1$ for all $l\gs l_0$; in each case, this can be shown by subtracting the denominator from the numerator, and checking that all derivatives of the resulting polynomial are positive at $l=l_0$. We deduce $r_{l_0}<r_{l_0+1}<\cdots$, and we can check directly that $r_{l_0}>1$ to get the result.

\begin{lemma}\label{deglem1}
Suppose $p=3$, $l\gs3$ and define
\[
\la_{(l)}=(3l+1\df10,\ 6,4,3,1).
\]
Then $\la_{(l)}$ is not homogeneous.
\end{lemma}

\begin{proof}
Define
\[
\mu_{(l)}=(3l+4,\ 3l-2\df7,\ 3,1)
\]
for each $l\gs3$. Then $\la_{(l)}^\reg=\mu_{(l)}^\reg$, and the bar-length formula gives
\[
\frac{\ddeg{\la_{(l+1)}}\ddeg{\mu_{(l)}}}{\ddeg{\mu_{(l+1)}}\ddeg{\la_{(l)}}}=\frac{(l+1)(3l+1)(3l+8)}{l(3l+5)(3l+7)}>1.
\]
So
\[
\frac{\ddeg{\la_{(3)}}}{\ddeg{\mu_{(3)}}}<\frac{\ddeg{\la_{(4)}}}{\ddeg{\mu_{(4)}}}<\frac{\ddeg{\la_{(5)}}}{\ddeg{\mu_{(5)}}}\dots.
\]
The bar-length formula also gives $\ddeg{\la_{(3)}}=\ddeg{\mu_{(3)}}$, so we are done by \cref{basicdeg}, except in the case $l=3$, where we can achieve the same result using the partition $(13,7,4)$ instead of $\mu_{(3)}$.
\end{proof}

\begin{lemma}\label{deglem5}
Suppose $p=3$, $l\gs4$, and define
\begin{align*}
\la_{(l)}&=(3l,3l-3,3l-7,3l-8,\ 3l-12\df3),
\\
\mu_{(l)}&=(3l,3l-3,3l-5,\ 3l-9\df6,\ 2).
\end{align*}
Then $\ddeg{\la_{(l)}}>\ddeg{\mu_{(l)}}$.
\end{lemma}

\begin{proof}
The bar-length formula gives
\[
\frac{\ddeg{\la_{(l+1)}}\ddeg{\mu_{(l)}}}{\ddeg{\mu_{(l+1)}}\ddeg{\la_{(l)}}}=\frac
{(l-3)l^3(2l-5)^2(2l-1)(3l-7)(3l-5)(3l-4)^2(3l+5)(6l-11)^2(6l-7)(6l+1)}
{(l-2)^4(l+2)(2l-3)^2(3l-8)(3l-1)(3l+1)^2(6l-17)(6l-13)(6l-5)^2(6l-1)}.
\]
This fraction is always greater than $1$ when $l\gs5$. Since $\ddeg{\la_{(l)}}>\ddeg{\mu_{(l)}}$ for $l=4,5$, the same is true for all $l\gs4$.
\end{proof}

\begin{lemma}\label{deglem6}
Suppose $p=3$, $l=4$ or $l\gs 7$, and define
\begin{align*}
\la_{(l)}&=(3l+1,3l-3,3l-7,3l-8,\ 3l-12\df3),
\\
\mu_{(l)}&=(3l+1,3l-3,3l-5,\ 3l-9\df6,\ 2)
\end{align*}
Then $\ddeg{\la_{(l)}}>\ddeg{\mu_{(l)}}$.
\end{lemma}

\begin{proof}
For $l=4$ the lemma can be easily checked. For $l\ge 7$ the bar-length formula gives
\[
\frac{\ddeg{\la_{(l+1)}}\ddeg{\mu_{(l)}}}{\ddeg{\mu_{(l+1)}}\ddeg{\la_{(l)}}}=\frac
{(l-3)(l-1)^2l(l+2)(2l-5)^2(3l-7)(3l-5)(3l-1)^2(3l+1)(3l+4)(6l-11)^2(6l-7)}
{(l-2)^4(l+1)^2(2l-3)(3l-8)(3l-2)^3(3l+7)(6l-17)(6l-13)(6l-5)(6l-1)}.
\]
Since this fraction is always greater than $1$ and $\ddeg{\la_{(7)}}>\ddeg{\mu_{(7)}}$, the result follows.
\end{proof}

\begin{lemma}\label{deglem7}
Suppose $p=3$, $l\gs6$, and define
\begin{align*}
\la_{(l)}&=(3l-1,3l-2,3l-7,3l-8,\ 3l-12\df3),
\\
\mu_{(l)}&=(3l-1,3l-2,3l-6,3l-8\ 3l-12\df6,\ 2).
\end{align*}
Then $\ddeg{\la_{(l)}}>\ddeg{\mu_{(l)}}$.
\end{lemma}

\begin{proof}
The bar-length formula gives
\[
\frac{\ddeg{\la_{(l+1)}}\ddeg{\mu_{(l)}}}{\ddeg{\mu_{(l+1)}}\ddeg{\la_{(l)}}}=\frac
{(l-4)(l-1)^3(l+1)(2l-5)^2(3l-10)(3l-8)(3l-4)(3l-1)(3l+2)(6l-1)} 
{(l-3)(l-2)^2l^3(2l-1)(3l-11)^2(3l-5)(3l+5)(6l-13)(6l-7)}>1.
\]
Since $\ddeg{\la_{(6)}}>\ddeg{\mu_{(6)}}$, the result follows.
\end{proof}

\begin{lemma}\label{deglem2}
Suppose $p=3$, $l\gs3$ and define
\[
\la_{(l)}=(3l\df3).
\]
Then $\la_{(l)}$ is not homogeneous.
\end{lemma}

\begin{proof}
Define
\[
\mu_{(l)}=(3l-1,3l-2,\ 3l-6\df3).
\]
Then $\mu_{(l)}^\reg=\la_{(l)}^\reg$, and the bar-length formula gives
\[
\frac{\ddeg{\la_{(l+1)}}\ddeg{\mu_{(l)}}}{\ddeg{\mu_{(l+1)}}\ddeg{\la_{(l)}}}=\frac{l^2(6l-5)(6l-1)}{(2l-1)^2(3l-2)(3l+2)}>1.
\]
Since $\ddeg{\la_{(3)}}>\ddeg{\mu_{(3)}}$, we are done by \cref{basicdeg}.
\end{proof}

\begin{lemma}\label{deglem4}
Suppose $p=3$, $l\gs7$, and define
\[
\la_{(l)}=(3l-1,3l-2,\ 3l-6\df3).
\]
Then $\la$ is not homogeneous.
\end{lemma}

\begin{proof}
Define
\[
\mu_{(l)}=(3l-1,3l-2,3l-7,3l-8,\ 3l-12\df3).
\]
Then $\mu_{(l)}^\reg=\la_{(l)}^\reg$, and
\[
\frac{\ddeg{\la_{(l+1)}}\ddeg{\mu_{(l)}}}{\ddeg{\mu_{(l+1)}}\ddeg{\la_{(l)}}}=\frac{(l-2)^2(2l-1)^2(6l-17)(6l-13)(6l-11)(6l-7)}{(2l-5)^2(2l-3)^2(3l-8)(3l-4)(6l-5)(6l-1)}>1.
\]
Now the fact that $\ddeg{\la_{(7)}}>\ddeg{\mu_{(7)}}$ gives the result.
\end{proof}

\begin{lemma}\label{deglem8}
Suppose $p=3$, $l\gs2$, and define
\[
\la_{(l)}=
\begin{cases}
(3l,3l-4,3l-5,\ 3l-9\df3)&\text{if }l\gs3
\\
(6,2,1)&\text{if }l=2.
\end{cases}
\]
Then $\la$ is not homogeneous.
\end{lemma}

\begin{proof}
We aim to find a strict partition $\mu$ such that $\mu^\reg=\la^\reg$ and $\ddeg\mu<\ddeg\la$. For $2\ls l\ls 7$ the partition $(3l-1,3l-2,\ 3l-6\df3)$ will do the trick, while for $l\gs8$ we define
\[
\mu_{(l)}=(3l,3l-3,3l-5,\ 3l-9\df6,\ 2).
\]
Then $\mu_{(l)}^\reg=\la_{(l)}^\reg$, and
\[
\frac{\ddeg{\la_{(l+1)}}\ddeg{\mu_{(l)}}}{\ddeg{\mu_{(l+1)}}\ddeg{\la_{(l)}}}=\frac
{(l-3)l^3(2l-3)^2(2l+1)(3l-7)(3l-5)(3l-2)^2(3l-1)(3l+5)} 
{(l-2)(l-1)^3(l+2)(2l-1)^2(3l-8)^2(3l+1)^3(6l-7)}>1.
\]
Now the fact that $\ddeg{\la_{(8)}}>\ddeg{\mu_{(8)}}$ gives the result.
\end{proof}

\begin{lemma}\label{deglem9}
Suppose $p=3$, $l\gs1$, and define
\begin{align*}
\la_{(l)}&=(6l+6,\ 6l+4\df3l+7,\ 3l+3,\ 3l+1\df4),
\\
\mu_{(l)}&=(6l+6,\ 6l+4\df3l+4,\ 3l,\ 3l-2\df4).
\end{align*}
Then $\ddeg{\la_{(l)}}>\ddeg{\mu_{(l)}}$.
\end{lemma}

\begin{proof}
The bar-length formula gives
\[
\frac{\ddeg{\la_{(l)}}}{\ddeg{\mu_{(l)}}}=\frac
{(l+1)(3l-1)(6l+7)(9l+8)(9l+10)}
{3l(l+2)(6l+1)(9l+7)^2}>1.\qedhere
\]
\end{proof}

\begin{lemma}\label{deglem10}
Suppose $p=3$, $l\gs6$, and define
\[
\la_{(l)}=(6l-4\df3l+5,\ 3l+2,3l,\ 3l-4\df2).
\]
Then $\la_{(l)}$ is not homogeneous.
\end{lemma}

\begin{proof}
Define
\[
\mu_{(l)}=(6l-4\df3l+5,\ 3l+3,3l-1,\ 3l-4\df2).
\]
Then $\mu_{(l)}^\reg=\la_{(l)}^\reg$, and
\[
\frac{\ddeg{\la_{(l)}}}{\ddeg{\mu_{(l)}}}=\frac{(l+1)(3l-4)(6l-1)(9l-1)}{(l-1)(3l-1)(6l+5)(9l-2)},
\]
which is greater than $1$ for $l\gs6$.
\end{proof}

\begin{lemma}\label{deglem11}
Suppose $p=3$, $l\gs 3$, and define
\begin{align*}
\la_{(l)}&=(3l,\ 3l-2\df4,\ 3,1),
\\
\mu_{(l)}&=\begin{cases}
(3l+1\df10,\ 6,4,3,1)&\text{if }l\gs4
\\
(13,7,4)&\text{if }l=3.
\end{cases}
\end{align*}
Then $\ddeg{\la_{(l)}}>\ddeg{\mu_{(l)}}$.
\end{lemma}

\begin{proof}
The lemma easily holds for $l=3$. So assume now that $l\gs 4$. Then the bar-length formula gives
\[
\frac{\ddeg{\la_{(l+1)}}\ddeg{\mu_{(l)}}}{\ddeg{\mu_{(l+1)}}\ddeg{\la_{(l)}}}=\frac{l(3l+7)(3l+10)(6l+5)}{(l+2)(3l+2)(3l+11)(6l+1)},
\]
which is greater than $1$ for $l\gs4$. Since $\ddeg{\la_{(4)}}>\ddeg{\mu_{(4)}}$ the result follows.
\end{proof}

\begin{lemma}\label{deglem12}
Suppose $p=3$, $l\gs1$, and define
\begin{align*}
\la_{(l)}&=(3l,\ 3l-2\df1),
\\
\mu_{(l)}&=(3l+1\df4).
\end{align*}
Then $\ddeg{\la_{(l)}}\geq \ddeg{\mu_{(l)}}$, with equality \iff $l=1$.
\end{lemma}

\begin{proof}
The bar-length formula gives
\[
\frac{\ddeg{\la_{(l)}}}{\ddeg{\mu_{(l)}}}=\frac{(3l+5)(3l+8)\dots(6l-1)}{(3l+4)(3l+7)\dots(6l-2)},
\]
which is obviously greater than $1$ when $l>1$.
\end{proof}

Now we prove a more general result using some of the lemmas in this section.

\begin{propn}\label{0211deg}
Suppose $p=3$, $l\gs3$, and
\[
\la=(3l-1+a_1,3l-4+a_2,\dots,-1+a_{l+1}),
\]
where:
\begin{itemize}
\item
$a_r\in\{0,1,2\}$ for each $r$;
\item
there is at least one $r\gs4$ for which $a_r\neq1$;
\item
if $a_r=2$ and $r\gs2$ then $a_{r-1}=0$;
\item
if $a_r=0$ then $r\ls l$ and $a_{r+1}=2$.
\end{itemize}
Then $\la$ is not homogeneous.
\end{propn}

\begin{proof}
We claim that there is a partition $\mu$ such that $\mu\doms\la$, $\mu_1=\la_1$, $\mu^\reg=\la^\reg$ and $\ddeg\la>\ddeg\mu$; \cref{basicdeg} then implies that $\la$ is not homogeneous.

We prove our claim by induction on $l$. If there is $r\gs5$ for which $a_r\neq1$ then by induction we can assume that the claim holds with $\la$ replaced by $(\la_2,\la_3,\dots)$, and we can use \cref{l419}. So we assume instead that $a_r=1$ for all $r\gs5$, while $a_4\neq1$. The tuple $(a_1,\dots,a_{l+1})$ must therefore be one of $(1,1,0,2,1,1,\dots,1)$, $(2,1,0,2,1,1,\dots,1)$ or $(0,2,0,2,1,1,\dots,1)$. In these cases the claim follows from \cref{deglem5,deglem6,deglem7} respectively (except for the small cases $\la=(9,6,2,1)$, $(10,6,2,1)$, $(16,12,8,7,3)$, $(19,15,11,10,6,3)$, $(8,7,2,1)$, $(11,10,5,4)$, $(14,13,8,7,3)$, where we can take $\mu=(9,6,3)$, $(10,6,3)$, $(16,12,9,7,2)$, $(19,15,12,10,6,2)$, $(8,7,3)$, $(11,10,6,3)$, $(14,13,9,6,3)$ respectively).
\end{proof}

\subsection{Projective modules}

Now we describe a technique we will use to show that certain partitions are not homogeneous using projective modules. The basic lemma we use is the following; for this lemma, we can take $p$ to be any odd prime.

\begin{lemma}\label{projlem}
Suppose $\la$ and $\nu$ are strict partitions with $\la^\reg=\nu^\reg$.
\begin{enumerate}
\item
If there is a projective supermodule $P$ such that
\[
\frac{[P:\s^\la]}{[\P^{\la^\reg}:\s^\la]}>\frac{[P:\s^\nu]}{[\P^{\la^\reg}:\s^\nu]},
\]
then $\la$ is not homogeneous.
\item
If $\ddeg\la>\ddeg\nu$ then there is a projective supermodule $P$ such that 
\[
\frac{[P:\s^\la]}{[\P^{\la^\reg}:\s^\la]}>\frac{[P:\s^\nu]}{[\P^{\la^\reg}:\s^\nu]}.
\]
\end{enumerate}
\end{lemma}

\needspace{3em}
\begin{proof}\indent
\begin{enumerate}
\vspace*{-\topsep}
\item
If $\la$ is homogeneous then by Brauer reciprocity $[\P^{\mu}:\s^{\la}]=0$ when $\mu\not=\la^\reg$. So if $P$ is any projective supermodule and we write $P=\bigoplus_\mu(\P^\mu)^{\oplus c_\mu}$, then
\begin{align*}
[P:\s^\la]=c_{\la^\reg}[\P^{\la^\reg}:\s^\la],
\\
[P:\s^\nu]\gs c_{\la^\reg}[\P^{\la^\reg}:\s^\nu],
\end{align*}
giving
\[
\frac{[P:\s^\la]}{[\P^{\la^\reg}:\s^\la]}\ls\frac{[P:\s^\nu]}{[\P^{\la^\reg}:\s^\nu]}.
\]
\item
Recall that
\[
\ddeg\la=\frac{\dim\s^\la}{[\s^\la:\D^{\la^\reg}]}=\frac{\sum_\mu[\s^\la:\D^\mu]\dim\D^\mu}{[\s^\la:\D^{\la^\reg}]}.
\]
So if $\ddeg\la>\ddeg\nu$ then there is some $\mu$ for which
\[
\frac{[\s^\la:\D^\mu]}{[\s^\la:\D^{\la^\reg}]}>\frac{[\s^\nu:\D^\mu]}{[\s^\nu:\D^{\la^\reg}]}.
\]
By Brauer reciprocity, this is the same as saying
\[
\frac{[\P^\mu:\s^\la]}{[\P^{\la^\reg}:\s^\la]}>\frac{[\P^\mu:\s^\nu]}{[\P^{\la^\reg}:\s^\nu]}.\qedhere
\]
\end{enumerate}
\end{proof}

The way we will apply \cref{projlem} is as follows. Suppose we have a strict partition $\la$ that we want to prove is not homogeneous. We will find another strict partition $\nu$ with $\la^\reg=\nu^\reg$, and a pair of smaller partitions $\la_-$ and $\nu_-$ with $\la_-^\reg=\nu_-^\reg$ and $\ddeg{\la_-}>\ddeg{\nu_-}$. Using part (ii) of \cref{projlem}, we will take a projective module $Q$ for which 
\[
\frac{[Q:\s^{\la_-}]}{[\P^{\la_-^\reg}:\s^{\la_-}]}>\frac{[Q:\s^{\nu_-}]}{[\P^{\la_-^\reg}:\s^{\nu_-}]},
\]
and apply a suitable induction functor to obtain a projective module $P$ with
\[
\frac{[P:\s^\la]}{[\P^{\la^\reg}:\s^\la]}>\frac{[P:\s^\nu]}{[\P^{\la^\reg}:\s^\nu]}.
\]
Then part (i) of \cref{projlem} gives the desired result.

The specific cases where we use this technique are in the following \lcnamecref{projcases}.

\begin{propn}\label{projcases}
Suppose $p=3$, and $\la$ is one of the partitions
\[
(9,7,4),(12,10,7,4),(15,13,10,7,4),(14,13,9,6,3),(17,16,12,9,6,3),(11,10,6,3).
\]
Then $\la$ is not homogeneous.
\end{propn}

\begin{proof}
We begin with the first three cases. So let $\la=(3l,\ 3l-2\df4)$ with $3\ls l\ls 5$, and set $\nu=(3l+1\df7,\ 3)$. Then $\la^\reg=\nu^\reg$ and $[\P^{\la^\reg}:\s^\la]=[\P^{\la^\reg}:\s^\nu]$ by \cref{L300921}. Also define $\mu^i=(3l-1\df2)+\de_i$ for $1\ls i\ls l$, where $\delta_i$ is the composition which has a $1$ in position $i$ and $0$ elsewhere. It is easy to check that $(\mu^1)^\reg=(\mu^l)^\reg$ and $\ddeg{\mu^l}>\ddeg{\mu^1}$. In addition, \cref{L300921} shows that $[\P^{(\mu^1)^\reg}:\s^{\mu^1}]=[\P^{(\mu^1)^\reg}:\s^{\mu^l}]$, so by \cref{projlem}(ii) we can find a projective module $Q$ with $[Q:\s^{\mu^l}]>[Q:\s^{\mu^1}]$. Now define $P=\Ind_0^{2l-2}Q$. Then
\begin{align*}
[P:\s^\la]&=\sum_{i=1}^l[Q:\s^{\mu^i}][\Ind_0^{2l-2}\s^{\mu^i}:\s^\la],\\
[P:\s^\nu]&=\sum_{i=1}^l[Q:\s^{\mu^i}][\Ind_0^{2l-2}\s^{\mu^i}:\s^\nu].
\end{align*}
The branching rule shows that if $1<i<l$ then $[\Ind_0^{2l-2}\s^{\mu^i}:\s^\la]=[\Ind_0^{2l-2}\s^{\mu^i}:\s^\nu]$ and
\[
[\Ind_0^{2l-2}\s^{\mu^l}:\s^\la]=[\Ind_0^{2l-2}\s^{\mu^1}:\s^\nu]>[\Ind_0^{2l-2}\s^{\mu^1}:\s^\la]=[\Ind_0^{2l-2}\s^{\mu^l}:\s^\nu],
\]
so that $[P:\s^\la]>[P:\s^\nu]$. So by \cref{projlem}(i) $\la$ is not homogeneous.

Next consider the last two cases, so $\la=(3l-1,3-2,\ 3l-6\df3)$ with $5\leq l\leq 6$. Let $\nu=(3l-1,3l-3,3l-5,\ 3l-9\df3)$ and again $\mu^i=(3l-1\df2)+\de_i$. Then we can conclude in the same way using $\mu^2$, $\mu^3$ and $\Ind_0^{l-1}$ instead of $\mu^l$, $\mu^1$ and $\Ind_0^{2l-2}$ respectively.

Finally consider the case $\la=(11,10,6,3)$. Here we take a slightly different (and simpler) approach. We let $\nu=(11,10,7,2)$. Then $\P^{(8,5,3,2)}\sim\s^{(11,5,2)}\oplus(\s^{(8,5,3,2)})^{\oplus2}$ (looking at decomposition matrices or by \cite[Theorem 4.4]{Mu}). The branching rule shows that $[\Ind_0^3\Ind_1^3\Ind_0^6\P^{(8,5,3,2)}:\s^\nu]=0$, while $[\Ind_0^3\Ind_1^3\Ind_0^6\P^{(8,5,3,2)}:\s^\la]$ is a non-zero multiple of $[\P^{(8,5,3,2)}:\s^{(8,5,3,2)}]$. So \cref{projlem} applies to show that $\la$ is not homogeneous.
\end{proof}

\section{Wreath products and RoCK blocks}\label{wreathsec}

In this section we analyse certain wreath product algebras that arise in the work of Kleshchev and Livesey \cite{KL} on RoCK blocks, and use these to derive information on modules labelled by special partitions.

\subsection{Wreath products}

First we recall some more theory of superalgebras; let $\F$ be an arbitrary field throughout this section. For any vector superspace $V$ and $0\not=v$ a homogeneous vector in $V$ we let $\overline{v}\in\{0,1\}$ be the parity of $v$. We use the usual tensor product rule for superalgebras $A$ and $B$: if $a_1,a_2\in A$ and $b_1,b_2\in B$ are all homogeneous, then $(a_1\otimes b_1)(a_2\otimes b_2)=(-1)^{\ol{b_1}\ol{a_2}}(a_1a_2)\otimes(b_1b_2)$ in $A\otimes B$.

We will be mainly concerned with the superalgebra $A_\bbf=\F[u]/(u^3)$, where the generator $u$ is homogeneous of odd degree. We usually write write $A_\bbf$ just as $A$ if $\bbf$ is understood. For any $d\gs0$ we define the wreath superproduct $A\wrs\sym_d$ as in \cite[\S2.2a]{KL}: as a vector superspace this is $A^{\otimes d}\otimes\F\sym_d$ with $\sym_d$ in even degree. The multiplication rule is defined by the tensor product rule above for multiplying in $A^{\otimes d}$, together with
\[
(1^{\otimes d};\si)(v_1\otimes\dots\otimes v_d;\pi)=(-1)^{[\si;v_1,\dots,v_d]}(v_{\si^{-1}(1)}\otimes\dots\otimes v_{\si^{-1}(d)};\si\pi),
\]
for $\si,\pi\in\sym_d$ and homogeneous $v_1,\dots,v_d\in A$, where
\[
[\si;v_1,\dots,v_d]=|\{(i,j)\mid1\leq i<j\leq d,\ \si(i)>\si(j),\ \overline{v_i}=1=\overline{v_j}\}|.
\]

Our aim in this section is to develop a small part of the representation theory of $\awrd$ (mimicking the work of Chuang and Tan \cite{CT}) to be able to apply the results from \cite{KL}; in particular, we will be interested in finding bounds for the diagonal Cartan invariants of $\awrd$.

We begin by constructing some $\awrd$-supermodules. (We will view $\awrd$ as a superalgebra, although as we shall see every simple supermodule is of type M, so it makes little difference whether we consider modules or supermodules.) For any $A$-supermodule $V$ and $\F\sym_d$-module $W$ (viewed as a supermodule concentrated in degree $0$) we define a module $T(V,W)$, which is $V^{\otimes d}\otimes W$ as a vector superspace, with the action of $A\wrs\sym_d$ given by
\begin{align*}
(1^{\otimes(i-1)}\otimes u\otimes1^{\otimes(d-i)};1)v_1\otimes\dots\otimes v_d\otimes w&=(-1)^{\sum_{j=1}^{i-1}\overline{v_d}}v_1\otimes v_{i-1}\otimes\dots\otimes uv_i\otimes v_{i+1}\otimes\dots\otimes v_d\otimes w\\
(1\otimes\dots\otimes1;\si)v_1\otimes\dots\otimes v_d\otimes w&=(-1)^{[\si;v_1,\dots,v_d]}v_{\si^{-1}(1)}\otimes\dots\otimes v_{\si^{-1}(d)}\otimes \si w
\end{align*}
for homogeneous $v_1,\dots,v_d\in V$, $w\in W$ and $\si\in\sym_d$.

This construction allows us to classify the simple supermodules for $\awrd$: let $J(R)$ denote the Jacobson radical of a ring $R$. It is clear that any element of $\awrd$ of the form $(1^{\otimes(i-1)}\otimes u\otimes1^{\otimes(d-i)};1)$ or of the form $(1^{\otimes d};m)$ with $m\in J(\F\sym_d)$ lies in $J(A\wrs\sym_d)$. Hence $\awrd/J(\awrd)\cong\F\sym_d/J(\F\sym_d)$, and every simple $\awrd$-module arises by taking a simple $\F\sym_d$-module and letting $u$ act as $0$. In other words, the simple $\awrd$-modules are the modules $T(\F,S)$ as $S$ ranges over the simple $\F\sym_d$-modules, where $\F$ is regarded as an $A$-module with $u$ acting as $0$. In particular, if $\nchar\F=0$ then the simple $\F\sym_d$-modules are the \emph{Specht modules} $\spe\la$ for partitions $\la$ of $d$, while if $\nchar\F=p>0$ then the simple $\F\sym_d$-modules are the \emph{James modules} $\jms\la$ for $p$-regular partitions $\la$.

We will adapt techniques from \cite{CT} to find the composition factors of $T(A,\spe\de)$ in characteristic $0$, which will enable us to find the diagonal Cartan invariants. First we introduce another construction: if $M$ is an $\awrd$-module and $W$ an $\F\sym_d$-module then we define the $\awrd$-module $M\oslash W$ to be the vector space $M\otimes W$ with
\[
(u_1\otimes\dots\otimes u_d;\si)(m\otimes w)=((u_1\otimes\dots\otimes u_d;\si)m)\otimes \si w.
\]
It is clear from the definition that $T(V,W)\oslash W'\cong T(V,W\otimes W')$.

The regular $A$-module has three composition factors all isomorphic to $\F$; we will use this to find filtrations and composition factors of the modules $T(A,\spe\nu)$. Given partitions $\al,\be,\ga,\nu$ with $|\al|+|\be|+|\ga|=|\nu|$, let $c^\nu_{\al\be\ga}$ denote the corresponding Littlewood--Richardson coefficient, i.e.\ the multiplicity of $\spe\nu$ as a composition factor of $\spe\al\otimes\spe\be\otimes\spe\ga\uparrow_{\sym_{|\al|,|\be|,|\ga|}}^{\sym_|\nu|}$ in characteristic $0$; here and below, $\sym_{a,b,c}$ denotes the Young subgroup $\sym_a\times\sym_b\times\sym_c\ls\sym_{a+b+c}$. For a partition $\be$ let $\be'$ denote the conjugate (or transpose) partition.

Now we can give the composition factors of the modules of the modules $T(A,\spe\nu)$ in characteristic~$0$.

\begin{lemma}\label{L210322_7}
Suppose $\nchar\F=0$, and let $\nu,\pi$ be partitions of $d$. Then
\[
[T(A,\spe\nu):T(\F,\spe\pi)]=\sum_{\substack{\al,\be,\ga\in\Par\\|\al|+|\be|+|\ga|=d}}c^\nu_{\al\be\ga}c^\pi_{\al\be'\ga}.
\]
\end{lemma}

\begin{proof}
We adapt the results in \cite[\S4]{CT} to wreath superproducts, specialising to the case of $A\wrs\sym_d$.

Because $\spe{(d)}$ is the trivial $\sym_d$-module, we can write $T(A,\spe\nu)=T(A,\spe{(d)})\oslash\spe\nu$. So if $0\leq M_1\leq\dots\leq M_h=T(A,\spe{(d)})$ is a filtration of $T(A,\spe{(d)})$, then
\[
0\leq M_1\oslash\spe\nu\leq\dots\leq M_h\oslash\spe\nu
\]
is a filtration of $T(A,\spe\nu)$. We thus start by finding a filtration of $T(A,\spe{(d)})$. As vector spaces we identify $T(A,\spe{(d)})$ and $A^{\otimes d}$.

The vector space $A^{\otimes d}$ has a basis $U=\{u^{e_1}\otimes\dots\otimes u^{e_d}\mid0\leq e_i\leq 2\text{ for each }i\}$. For any composition $(a,b,c)$ of $d$ let $U_{a,b,c}$ be the subset of $U$ consisting of all basis elements $u^{e_1}\otimes\dots\otimes u^{e_d}$ in which $1,u,u^2$ appear $a,b,c$ times respectively.

From the filtration $0\subseteq\langle u^2\rangle\subseteq\langle u,u^2\rangle\subseteq A$ of $A$ we have a filtration
\[0=M_0\leq\dots\leq M_{2d+1}=T(A,\spe{(d)})\]
where $M_i=\langle U_{a,b,c}\mid b+2c\geq 2d+1-i\rangle$. For any $(a,b,c)$ with $b+2c=2d+1-i$, let $V_{a,b,c}= (\langle U_{a,b,c}\rangle_\F\oplus M_{i-1})/M_{i-1}$. Then
\[
\frac{M_i}{M_{i-1}}\cong \bigoplus_{\substack{a,b,c\gs0\\b+2c=2d+1-i}}V_{a,b,c},
\]
so that $T(A,\spe{(d)})$ is filtered by the modules $V_{a,b,c}$ as $(a,b,c)$ ranges over compositions of $d$. We will show that
\begin{equation}\label{E210322}
V_{a,b,c}\cong T\left(\F,(\spe{(a)}\otimes\spe{(1^b)}\otimes\spe{(c)})\ua_{\sym_{a,b,c}}^{\sym_d}\right).
\end{equation}
It is easy to see that any vector of the form $(1^{\otimes(i-1)}\otimes u\otimes1^{\otimes(d-i)};1)$ acts as 0 on both modules. So we only need to compare the action of the elements $(1^{\otimes d};\si)$.

Let $v$ be the image of $1^{\otimes a}\otimes u^{\otimes b}\otimes (u^2)^{\otimes c}$ in $V_{a,b,c}$ and $R$ be a set of representatives of $\sym_d/\sym_{a,b,c}$. Then $\{(1^{\otimes d};\pi)v\mid \pi\in R\}$ is a basis of $V_{a,b,c}$. To work out the action of $\sym_d$ on this basis, take $\si\in\sym_d$ and $\pi\in R$, and let $\rho\in R$ be such that $\si\pi\in\rho\sym_{a,b,c}$. Then $\rho^{-1}\si\pi\in\sym_{a,b,c}$; we let $\ol\si$ be the permutation induced by $\rho^{-1}\si\pi$ on $\{a+1,\dots,a+b\}$.

Then
\[
(1^{\otimes d};\si)(1^{\otimes d};\pi)v=(1^{\otimes d};\rho)(1^{\otimes d};\rho^{-1}\si\pi)v=(1^{\otimes d};\rho)(-1)^{\sgn(\overline{\si})}v=(-1)^{\sgn(\overline{\si})}(1^{\otimes d};\rho)v.
\]
This action coincides with the action of $\sym_d$ on $(\spe{(a)}\otimes\spe{(1^b)}\otimes\spe{(c)})\ua_{\sym_{a,b,c}}^{\sym_d}$ (with basis $\{\pi w\mid\pi\in R\}$ where $w$ is any non-zero element of $\spe{(a)}\otimes\spe{(1^b)}\otimes\spe{(c)}$), so that \eqref{E210322} follows.

Thus $T(A,\spe{(d)})$ has a filtration with a factor isomorphic to $T(\F,(\spe{(a)}\otimes\spe{(1^b)}\otimes\spe{(c)})\ua_{\sym_{a,b,c}}^{\sym_d})$ for each composition $(a,b,c)$ of $d$. Hence $T(A,\spe\nu)\cong T(A,\spe{(n)})\oslash\spe\nu$ has a filtration with a factor isomorphic to
\begin{align*}
T(\F,(\spe{(a)}\otimes\spe{(1^b)}\otimes\spe{(c)})\ua_{\sym_{a,b,c}}^{\sym_d})\oslash\spe\nu&\cong T(\F,((\spe{(a)}\otimes\spe{(1^b)}\otimes\spe{(c)})\ua_{\sym_{a,b,c}}^{\sym_d})\otimes\spe\nu)\\
&\cong T(\F,((\spe{(a)}\otimes\spe{(1^b)}\otimes\spe{(c)})\otimes\spe\nu\da^{\sym_d}_{\sym_{a,b,c}})\ua_{\sym_{a,b,c}}^{\sym_d})
\end{align*}
for each $(a,b,c)$. Now
\begin{align*}
(\spe{(a)}\otimes\spe{(1^b)}\otimes\spe{(c)})\otimes\spe\nu\da^{\sym_d}_{\sym_{a,b,c}}
&\cong\bigoplus_{\substack{\al\in\Par(a)\\\be\in\Par(b)\\\ga\in\Par(c)}}(\spe{(a)}\otimes\spe{(1^b)}\otimes\spe{(c)})\otimes(\spe\al\otimes\spe\be\otimes\spe\ga)^{\oplus c^\nu_{\al\be\ga}}
\\
&\cong\bigoplus_{\substack{\al\in\Par(a)\\\be\in\Par(b)\\\ga\in\Par(c)}}(\spe\al\otimes\spe{\be'}\otimes\spe\ga)^{\oplus c^\nu_{\al\be\ga}}
\end{align*}
so that
\[
\left[T(\F,(\spe{(a)}\otimes\spe{(1^b)}\otimes\spe{(c)})\ua_{\sym_{a,b,c}}^{\sym_d})\oslash\spe\nu:T(\F,\spe\pi)\right]=\sum_{\substack{\al\in\Par(a)\\\be\in\Par(b)\\\ga\in\Par(c)}}c^\nu_{\al\be\ga}c^\pi_{\al\be'\ga},
\]
and the result follows.
\end{proof}

Now we turn to Cartan invariants.

\begin{lemma}\label{L210322_6}
The projective $\awrd$-modules are precisely the modules $T(A,P)$, for $P$ a projective $\F\sym_d$-module. If $S$ is a simple $\F\sym_d$-module with projective cover $P(S)$, then $T(A,P(S))$ is the projective cover of $T(\F,S)$.
\end{lemma}

\begin{proof}
As $T(A,\F\sym_d)\cong A\wrs\sym_d$ and $T(A,P\oplus Q)=T(A,P)\oplus T(A,Q)$ it is clear that the modules $T(A,P)$ are projective, and we have a direct sum decomposition
\[
\awrd=\bigoplus_ST(A,P(S))^{\oplus\dim S},
\]
summing over all simple $\F\sym_d$-modules $S$. The simple $\awrd$-modules are the modules $T(\F,S)$, with $\dim T(\F,S)=\dim S$ for each $S$. So the summands in the expression above are indecomposable, and therefore every projective module occurs as $T(A,P)$ for some projective $P$. Moreover, it is easy to check that $T(\F,S)$ occurs as a quotient of $T(A,P(S))$, and therefore $T(A,P(S))$ is the projective cover of $T(\F,S)$.
\end{proof}

Now we can explicitly find the diagonal Cartan invariants for $\awrd$ in characteristic $0$. Let $C=(c_{\nu,\pi})$ be the Cartan matrix of $A\wrs\sym_d$ with rows and columns indexed by partitions of $d$, so that $c_{\nu,\pi}$ is the multiplicity of $T(\F,\spe\nu)$ as a composition factor of the projective cover of $T(\F,\spe\pi)$.

\begin{lemma}\label{L210322_5}
Suppose $\nchar\F=0$ and $\nu\in\Par(d)$. Then $c_{\nu,\nu}\geq 2d+1$, with equality \iff $\ga=(d)$ or $(1^d)$.
\end{lemma}

\begin{proof}
Because $\F\sym_d$ is semisimple, $\spe\nu$ is it own projective cover, and therefore $T(A,\spe\nu)$ is the projective cover of $T(\F,\spe\nu)$. So by \cref{L210322_7}
\[
c_{\nu,\nu}=[T(A,\spe\nu):T(\F,\spe\nu)]=\sum_{\substack{\al,\be,\ga\in\Par\\|\al|+|\be|+|\ga|=d}}c^\nu_{\al\be\ga}c^\nu_{\al\be'\ga}.
\]
We first recall the following simple fact about the Littlewood--Richardson coefficients: if $\be\subseteq\nu$ and $0\ls i\ls d-|\be|$, there exist partitions $\al\in\Par(i)$ and $\ga\in\Par(d-|\be|-i)$ for which $c^\nu_{\al\be\ga}>0$.

Assume first that $\ga\not=(d),(1^d)$. Then $d\geq 3$ and $\varnothing,(1),(2,1)\subseteq \ga$. So (using the fact above)
\begin{align*}
c_{\ga,\ga}
&\geq\sum_{\substack{\al,\ga\in\Par\\|\al|+|\ga|=d}}(c^\nu_{\al\varnothing\ga})^2+\sum_{\substack{\al,\ga\in\Par\\|\al|+|\ga|=d-1}}(c^\nu_{\al(1)\ga})^2+\sum_{\substack{\al,\ga\in\Par\\|\al|+|\ga|=d-3}}(c^\nu_{\al(2,1)\ga})^2
\\
&\gs(d+1)+d+(d-2)
\\
&>2d+1.
\end{align*}
Now assume that $\ga=(d)$ (the case $\ga=(1^d)$ being similar). In this case we use the fact that
\[
c^{(d)}_{\al\be\ga}=
\begin{cases}
1&\text{if }\al=(|\al|),\ \be=(|\be|),\ \ga=(|\ga|)
\\
0&\text{otherwise}
\end{cases}
\]
and
\[
c^{(d)}_{\al\be'\ga}=
\begin{cases}
1&\text{if }\al=(|\al|),\ \be=(1^{|\be|}),\ \ga=(|\ga|)
\\
0&\text{otherwise},
\end{cases}
\]
which gives
\[
c_{(d),(d)}=\sum_{i=0}^d(c^{(d)}_{(i)\varnothing(d-i)})^2+\sum_{i=0}^{d-1}(c^{(d)}_{(i)(1)(d-i-1)})^2=(d+1)+d=2d+1.\qedhere
\]
\end{proof}

Now we prove a similar result in characteristic $3$. Given $\pi\in\Par(d)$, let $\ospe\pi$ denote a $3$-modular reduction of $\spe\pi$. Recall that if $\mu$ is a $3$-regular partition then $\jms\mu$ denotes the James module for $\sym_d$ in characteristic $3$. Let $\ol C$ denote the Cartan matrix for $\awrd$ in characteristic $3$, with entries $\ol c_{\mu,\pi}=[T(\F,P(\jms\pi))):T(\F,\jms\mu)]$.

\begin{lemma}\label{L210322_2}
Suppose $\nchar\F=3$, $d\geq 3$ and $\mu$ is a $3$-regular partition of $d$. Then $\ol c_{\ga,\ga}>2d+1$.
\end{lemma}

\begin{proof}
In this proof $A$ denotes the algebra $\bbf(u)/(u^3)$Given $\pi\in\Par(d)$, let $d_{\pi\mu}=[\ospe\pi:\jms\mu]$ be the corresponding decomposition number for $\F\sym_d$. Then it is clear that $d_{\pi\mu}=[T(\F,\ospe\pi):T(\F,\jms\mu)]$. So, taking the result of \cref{L210322_7} and reducing modulo $3$,
\[
[T(A,\ospe\nu):T(\F,\jms\mu)]=\sum_{\pi\in\Par(d)}[T(A_{\bbc},\spe\nu):T(\mathbb C,\spe\pi)]d_{\pi\mu}.
\]
So by Brauer reciprocity
\[
\ol c_{\mu,\mu}=[T(A,P(\jms\mu)):T(\F,\jms\mu)]=\sum_{\nu\in\Par(d)}d_{\nu\mu}\sum_{\pi\in\Par(d)}[T(A_{\mathbb C},\spe\nu):T(\mathbb C,\spe\pi)]d_{\pi\mu}.
\]
Taking just the terms with $\nu=\pi$ gives
\[
\ol c_{\mu,\mu}\gs\sum_{\nu\in\Par(d)}[T(A_{\mathbb C},\spe\nu):T(\mathbb C,\spe\nu)]d_{\nu\mu}^2=\sum_{\nu\in\Par(d)}c_{\mu,\mu}d_{\nu\mu}^2.
\]
If $\mu\neq(d)$ then $d_{\mu\mu}=1$ and $c_{\mu\mu}>2d+1$ by \cref{L210322_5}, so we are done (note that $\mu$ cannot equal $(1^d)$ because $d\gs3$ and $\mu$ is $3$-regular). In the case $\mu=(d)$ the assumption that $d\gs3$ means that there is a partition $\nu\neq\mu$ for which $d_{\nu\mu}>0$, and so
\[
\ol c_{\mu,\mu}\gs c_{\mu\mu}+c_{\nu\nu}>2d+1,
\]
again using \cref{L210322_5}.
\end{proof}

\subsection{\Spc partitions and RoCK blocks}

Now we consider \spc partitions more closely. Even though our main theorem excludes the case of \spc partitions, we still need to know something about modules labelled by \spc partitions in order to prove our main result. To do this, we combine our analysis of wreath products in \cref{wreathsec} with the results of Kleshchev and Livesey.

Recall that a strict partition is \spc if it has the form $\nu+3\al$ with $\nu$ a 3-bar core and $\al$ a partition with $l(\al)\leq l(\nu)$. In the case where $\la=\nu+3\al$ where the non-zero parts of $\nu$ are congruent to $1$ modulo $3$ and $l(\nu)\geq|\al|$, the block of $\calt_{|\la|}$ containing $\s^\la$ is a \emph{spin RoCK block} as defined in \cite{KL}. We first prove that when considering partitions of the above given form we may always assume that we are in a RoCK block.

\begin{lemma}\label{L210322_3}
Let $\al$ be a partition. If there exists a 3-bar core $\nu$ with $l(\nu)\geq l(\al)$ such that $\nu+3\al$ is homogeneous, then $\pi+3\al$ is homogeneous for all 3-bar cores $\pi$ with $l(\pi)\geq l(\al)$.
\end{lemma}

\begin{proof}
It is enough to prove that if $l\geq l(\la)$ and
\[
\pi=(3l-2\df1),\qquad\nu=(3l-1\df2),\qquad\psi=(3l+1\df1),
\]
then $\nu+3\al$ is homogeneous if and only if $\pi+3\al$ is, and similarly with $\psi$ in place of $\pi$.

It follows by the definition that all parts of $\nu+3\al$ are congruent to $2$ modulo $3$, so that $(\nu+3\al)^{-1}=\pi+3\al$ and $(\pi+3\al)^{+1}=\nu+3\al$. Similarly we find that $(\nu+3\al)^{+0}=\psi+3\al$ and $(\psi+3\al)^{-0}=\nu+3\al$. The lemma then follows from \cref{T300921,T300921_2}.
\end{proof}

Next we prove a column removal result for $\al$.

\begin{lemma}\label{L210322_4}
Suppose $\al$ is a partition. Let $l=l(\al)$ and
\[
\nu=(3l-2\df1),\qquad\pi=(3l-1\df2),\qquad\be=(\al_1-1,\dots,\al_l-1).
\]
If $\nu+3\al$ is homogeneous then so is $\pi+3\be$.
\end{lemma}

\begin{proof}
Note that all parts of $\nu+3\al$ are congruent to $1$ modulo $3$, so $\pi+3\be=(\nu+3\al)^{-0}$. The lemma then follows from \cref{T300921}.
\end{proof}

The converse does not necessarily hold, as the following \lcnamecref{T210322} shows.

\begin{theor}\label{T210322}
Let $\nu=(3l-2\df1)$, where $l\geq d\geq 3$. Then $\nu+(3^d)$ is not homogeneous.
\end{theor}

To prove \cref{T210322}, we will use results on RoCK blocks proved recently by Kleshchev and Livesey, specialising to the case $p=3$. So we fix $\nu=(3l-2\df1)$ and $d\ls l$, and let $\calb^{\nu,d}$ be the block of $\calt_{|\nu|+3d}$ with bar core $\nu$ and bar-weight $d$. (We will abuse notation by saying that a $3$-strict partition $\la$ lies in $\calb^{\nu,d}$ if it has $3$-bar core $\nu$ and $3$-bar-weight $d$.)

\begin{lemma}\label{3strictnu}
The $3$-strict partitions in $\calb^{\nu,d}$ are the partitions $(\nu+3\al)\sqcup3\be$, for partitions $\al,\be$ with $|\al|+|\be|=d$. Such a partition is restricted \iff $\al=\varnothing$.
\end{lemma}

\begin{proof}
The first statement is easily proved by induction on $d$: a partition in $\calb^{\nu,d}$ is obtained by adding a $3$-bar to a partition in $\calb^{\nu,d-1}$. By induction such a partition can be written in the form $(\nu+3\al^-)\sqcup3\be^-$ where $|\al^-|+|\be^-|=d-1$. Since $d-1<l$, the last non-zero part of this partition equals $1$. So the only way to add a $3$-bar is to increase one of the parts by $3$. Increasing one of the parts congruent to $1$ modulo $3$ corresponds to adding a node to $\al^-$, while increasing one of the parts divisible by $3$ corresponds to adding a node to $\be$. So we obtain a partition of the required form.

The second statement is easy to check.
\end{proof}

Now we compute the diagonal Cartan entry for the simple module $\D^{\nu\sqcup(3d)}$.

\begin{lemma}\label{L210322}
Let $\nu$ and $d$ be as above, and let $\mu=\nu\sqcup(3d)$. Then $[\P^\mu:\D^\mu]=2d+1$.
\end{lemma}

\begin{proof}
\cref{3strictnu} shows that $\mu$ is the most dominant restricted $3$-strict partition in $\calb^{\nu,d}$. So \cref{L300921} gives
\[
\P^\mu\sim\sum_{\substack{\la\in\RP\\\la^\reg=\mu}}[\P^{\la^\reg}:\s^\la]\s^\la,
\]
where the coefficient $[\P^{\la^\reg}:\s^\la]$ is given in \cref{L300921}. So to obtain the formula for $[\P^\mu]$ we need to find which strict partitions $\la$ satisfy $\la^\reg=\mu$. We claim that $\la^\reg=\mu$ \iff $\la$ has the form $(\nu+(3^{d-i}))\sqcup(3i)$ for $0\ls i\ls d$. The ``if'' part is easy to check by directly comparing ladders. For the ``only if'' part, assume $\la\in\calb^{\nu,d}$ and use \cref{3strictnu} to write $\la=(\nu+3\al)\sqcup 3\be$ for some partition $\al,\be$ with $|\al|+|\be|=d$. If $l(\be)\geq 2$ then $l(\la^\reg)\geq l(\la)\gs l+2>l(\mu)$, so that $\la^\reg\neq\mu$. If $\al_1\geq 2$ then $\la$ contains the node $(1,3l+4)$ which lies in ladder $2l+2$, so that $l(\la^\reg)\geq l+2>l(\mu)$, so again $\la^\reg\neq\mu$.

Hence
\[
[\P^\mu:\D^\mu]=\sum_{i=0}^d[\P^\mu:\s^{\nu+(3^{d-i})\sqcup(3i)}][\s^{\nu+(3^{d-i})\sqcup(3i)}:\D^\mu].
\]
Combining the two equations in \cref{L300921}, we find that if $\la^\reg=\mu$ then $[\P^\mu:\s^\la][\s^\la:\D^\mu]=2^{l_3(\la)}$. We can easily see that $l_3((\nu+(3^{d-i}))\sqcup(3i))$ equals $0$ if $i=0$, or $1$ if $i\gs1$. The result follows.
\end{proof}

Now we summarise the results we need from \cite{KL}. Keeping $\nu,d$ as above, the assumption that $d\ls l$ means that $\calb^{\nu,d}$ is a \emph{RoCK block}. In view of \cref{L210322} we can replace $l$ with a larger value, and we choose $l$ such that $l\equiv2d\Md4$. Now we claim that $n=|\nu|+3d$ is even and that the block $\calb^{\nu,d}$ is of type M. To see that $n$ is even, observe that $|\nu|=l(3l-1)/2\equiv\lceil l/2\rceil\Md2$, so if $l\equiv2d\Md4$, then $n=|\nu|+3d\equiv d+3d\equiv0\Md2$. To see that the block $\calb^{\nu,d}$ is of type M, observe that the number of $1$-nodes of $\nu$ is $l(l-1)/2$, so that the number of $1$-nodes of a partition with $3$-bar core $\nu$ and $3$-bar-weight $d$ is $d+l(l-1)/2\equiv d+\lfloor l/2\rfloor\Md2$.

Rather than treating $\calt_n$ directly, Kleshchev and Livesey work with the \emph{Sergeev superalgebra} $\caltc_n=\calt_n\otimes\calc_n$, where $\calc_n$ is the Clifford algebra on $n$ generators. The assumption that $n$ is even guarantees that $\caltc_n$ is Morita superequivalent to $\calt_n$. We let $\calbc^{\nu,d}$ be the (super)block corresponding to $\calb^{\nu,d}$. Now we need the following result, which relates RoCK blocks to the wreath products $\awrd$ studied in \cref{wreathsec}.

\begin{theor}\cite[Theorem F]{KL}\label{klthmf}
Suppose $\nu,d$ are as above, and let $r_0$ denote the number of nodes of $\nu$ of residue $0$. Then there is an idempotent $f\in\calbc^{\nu,d}$ such that
\[
f\calbc^{\nu,d}f\cong(\awrd)\otimes\calc_{r_0+2d}.
\]
\end{theor}

Our assumption that $l$ is even guarantees that the integer $r_0+2d$ is even, which means that $(\awrd)\otimes\calc_{r_0+2d}$ is Morita equivalent to $\awrd$. We want to examine which simple $\calbc^{\nu,d}$-modules are killed by the idempotent $f$.

\begin{propn}\label{fd0}
Let $\D$ be the simple supermodule corresponding to $\D^{\nu\sqcup(3d)}$ under the Morita equivalence between $\calb^{\nu,d}$ and $\calbc^{\nu,d}$. Then $f\D=0$.
\end{propn}

\begin{proof}
We prove \cref{fd0} by comparing diagonal Cartan entries. In general, if $e$ is an idempotent in an algebra $R$ and $M$ is a simple $R$ module such that $eM\neq0$, then the diagonal Cartan entry $[P(eM):eM]$ for $eRe$ is the same as the diagonal Cartan entry $[P(M):M]$ for $R$. By \cref{L210322} (and our assumption that the block $\calb^{\nu,d}$ is of type M) the Cartan number $[P(\D):\D]$ for $\calb^{\nu,d}$ is $2d+1$. On the other hand by \cref{L210322_2} any diagonal entry of the Cartan matrix of $\awrd$ is larger than $2d+1$. So $f\D$ must be zero.
\end{proof}

Next we want to consider the $\calbc^{\nu,d}$-module corresponding to $\s^{\nu+(3^d)}$.

\begin{propn}\label{fsnon0}
Let $\s$ be the supermodule corresponding to $\s^{\nu+(3^d)}$ under the Morita equivalence between $\calb^{\nu,d}$ and $\calbc^{\nu,d}$. Then $f\s\neq0$.
\end{propn}

\begin{proof}
To prove this, we need to examine the idempotent $f$ in more detail and consider weight spaces. The way $f$ is constructed means that if $M$ is a module then $fM\neq0$ \iff $M$ has a weight $i=(i_1,\dots,i_n)$ such that:
\begin{itemize}
\item
$(i_1,\dots,i_{|\nu|})$ is a weight of $s^\nu$;
\item
for $j=1,\dots,d$ the residues $i_{|\nu|+3j-2},i_{|\nu|+3j-1},i_{|\nu|+3j}$ equal $0,0,1$ in some order.
\end{itemize}
To show that $\s$ has such a weight, we use \cref{tabweight}, and we need to show that there is a standard shifted $(\nu+(3^d))$-tableau with $i^t$ of the above form. But this is straightforward: we take any standard shifted $\nu$-tableau, and then for $j=1,\dots,d$ map $|\nu|+3j-2,|\nu|+3j-1,|\nu|+3j$ to the nodes $(j,\nu_j+1),(j,\nu_j+2),(j,\nu_j+3)$.
\end{proof}

\begin{proof}[Proof of \cref{T210322}]
Let $\s$ be the supermodule corresponding to $\s^{\nu+(3^d)}$ under the Morita equivalence between $\calb^{\nu,d}$ and $\calbc^{\nu,d}$, and let $f\in\calbc^{\nu,d}$ be the idempotent described above. By \cref{fsnon0} $\s$ has a a composition factor $M$ such that $fM\neq0$. On the other hand, \cref{L300921,fd0} show that $\s$ also has a composition factor $\D=\D^{(\nu+(3^d))^\reg}=\D^{\nu\sqcup(3d)}$ such that $f\D=0$. So $\s$ has composition factors of at least two isomorphism types.
\end{proof}

\begin{cor}\label{C210322}
Let $\nu$ be a $3$-bar core and $\al=(a^b)$ with $b\leq l(\nu)$. Then $\nu+3\al$ is homogeneous if and only if $\al=(a)$ or $\al=(1^2)$.
\end{cor}

\begin{proof}
If $\al=(a)$ then $\nu+3\al$ is homogeneous by \cref{walthm,L210322_3}. If $\al=(1^2)$ we can argue again by \cref{L210322_3} using the fact that the composition factors of $\s^{(7,4)}$ are known.

Assume next that $b=2$ and $a\geq 2$. Since $(10,7)$ is not homogeneous (from known decomposition numbers), $\nu+3\al$ is also not homogeneous by \cref{L210322_3,L210322_4}. For the case where $b\geq 3$ and $a=1$ \cref{T210322} implies that $\nu+3\al$ is not homogeneous, and then the general case where $b\geq3$ follows using \cref{L210322_4}.
\end{proof}

\section{Proof of \cref{mainthm}}\label{proofmainthm}

In this final section we give the proof of \cref{T081021}, which is a reformulation of our main result \cref{mainthm}. We begin with some combinatorial lemmas we will need later. The proof of the following \lcnamecref{casefourlem} is an easy exercise.

\begin{lemma}\label{casefourlem}
For $l\gs1$ define the partitions
\[
\si(l)=(3l-2\df7,\ 6,4,3,1),\qquad\tau(l)=(3l-1\df8,\ 6,5,3,2).
\]
Then $\si(l)^{+1}=\tau(l)$ and $\tau(l)^{+0}=\si(l+1)$.
\end{lemma}

\begin{lemma}\label{L110222}
Suppose $\la$ is a strict partition with no parts congruent to $1$ modulo $3$. Then $(\la^{+0})^{+1}$ has no parts congruent to $1$ modulo $3$.
\end{lemma}

\begin{proof}
Assume for a contradiction that $(\la^{+0})^{+1}_r=3a+1$. Then $\la^{+0}_r=3a+1$ and $(r,3a+2)$ is not a \sadbl node of $\la^{+0}$. This can only happen if $r\gs2$ and $\la^{+0}_{r-1}=3a+2$. This in turn means that $\la_{r-1}=3a+2$ and $(r-1,3a+3)$ is not a \sadbl node of $\la$. This is only possible if $r\gs 4$ and $(\la_{r-3},\la_{r-2})=(3a+4,3a+3)$. But this contradicts the hypothesis.
\end{proof}

\begin{lemma}\label{L110222_2}
Suppose $\la$ is a strict partition with no parts congruent to $1$ modulo $3$. Then $l((\la^{+0})^{+1})=l(\la)+1$ and $(\la^{+0})^{+1}_{l((\la^{+0})^{+1})}=2$.
\end{lemma}

\begin{proof}
Let $l=l(\la)$. Then $\la^{+0}_{l+1}=1$ because $\la_l\gs2$ (or $l=0$), so that $(l+1,1)$ is a \sadbl node of $\la$. The result then follows from \cref{L110222}.
\end{proof}

\begin{lemma}\label{L110222_3}
Suppose $\la$ is a strict partition with no parts congruent to $1$ modulo $3$. Suppose $1\leq r<s\leq l(\la)$ and $x\gs1$ such that
\[
(\la_r,\dots,\la_s)=(3x,\ 3x-1\df3(x-r+s+1)-1)
\]
Then $(\la^{+0})^{+1}_r\gs3x+2$ and
\[
((\la^{+0})^{+1}_{r+1},\dots,(\la^{+0})^{+1}_s)=(3x,\ 3x-1\df3(x-r+s+2)-1).
\]
\end{lemma}

\begin{proof}
The first statement holds because $\la_{r-1}\gs3x+2$ if $r\geq2$. The second statement holds by \cref{L110222}.
\end{proof}

Now we turn to \cref{mainthm}. We fix $p=3$ for the remainder of the paper. Let
\begin{align*}
H_0&=\lset{\la}{\la\text{ is special}},
\\
H_1&=\lset{(3a)}{a\gs2},
\\
H_2&=\lset{\nu\sqcup(3)}{\nu\text{ a $3$-bar core}},
\\
H_3&=\{(2,1),(3,2,1),(4,3,2),(4,3,2,1),(5,3,2,1),(5,4,3,1),
\\
&\phantom{=\{}(5,4,3,2),(5,4,3,2,1),(7,4,3,2,1),(8,5,3,2,1)\}.
\end{align*}
First we deal with the ``if'' part, and check that the partitions i $H_1\cup H_2\cup H_3$ are homogeneous. For the partitions in $H_1$ this follows from \cref{walthm}, and for the partitions in $H_2$ we use \cref{mullerthm}. For partitions in $H_3$ we can use known decomposition numbers for all except the last partition $(8,5,3,2,1)$. In particular, we know that $(7,4,3,2,1)$ is homogeneous. Since  $(8,5,3,2,1)=(7,4,3,2,1)^{+1}$, the partition $(8,5,3,2,1)$ is also homogeneous in view of \cref{T300921_2}.

It remains to prove the ``only if'' part of the theorem. So we assume $\la$ is a strict partition with $\la\notin H_0\cup H_1\cup H_2\cup H_3$, and we will show that $\la$ is not homogeneous. We use induction on $n=|\la|$. So we assume the result is true for all partitions smaller than $\la$. Recall from \cref{T300921} that if there is $i\in\{0,1\}$ for which $\la^{-i}$ is not homogeneous, then $\la$ is not homogeneous. So if there is $i$ such that $\la^{-i}\notin H_0\cup H_1\cup H_2\cup H_3\cup\{\la\}$, then we are done by induction. So we assume for the rest of this section that $\la^{-0},\la^{-1}\in H_0\cup H_1\cup H_2\cup H_3\cup\{\la\}$. Because $\la$ must have a \srmbl $i$-node for some $i$, there is at least one $i$ for which $\la^{-i}\neq\la$. So we choose $i$ such that $\la^{-i}\neq\la$, and we let $a=n-|\la^{-i}|=\heps_i(\la)$. In view of \cref{T300921} we can assume that $\eps_i(\la^\reg)=a$ and $(\la^{-i})^\reg=(\la^\reg)\im=\te_i^a(\la^\reg)$, and therefore $\la^\reg=\tf_i^a((\la^{-i})^\reg)$.

With these assumptions in place, we consider several cases and subcases.

\case
\edef\onepartcase{\caselabel}
Suppose $\la^{-i}=(m)$ for some $m<n$. Then $i=0$ and $\la=(n-1,1)$ with $n\nequiv2\ppmod3$ and $n\gs6$. Now the theorem follows by \cref{walthm}.

\case
Suppose $\la^{-i}$ is one of the partitions $(2,1)$, $(3,2,1)$, $(4,3,2)$, $(4,3,2,1)$, $(5,3,2,1)$, $(5,4,3,1)$, $(5,4,3,2)$, $(5,4,3,2,1)$, $(7,4,3,2,1)$, $(8,5,3,2,1)$.

Since $\heps_i(\la^{-i})=0$ and $\hphi_i(\la^{-i})\geq1$, it follows that
\[
\la^{-i}=
\begin{cases}
(4,3,2)\text{ or }(5,4,3,2)&\text{if $i=0$}
\\
(4,3,2,1),(5,4,3,1)\text{ or }(7,4,3,2,1)&\text{if $i=1$.}
\end{cases}
\]
Since $\la^\reg=\tf_i^a((\la^{-i})^\reg)$ (and $\la^{-i}=(\la^{-i})^\reg$ in each of the above cases), it follows that
\begin{align*}
\la\in\,\{&(4,3,2,1),(5,4,3,2,1),(6,4,3,2,1),(7,4,3,2,1),(5,3,2,1),\\
&(5,4,3,2),(7,5,3,2,1),(8,4,3,2,1),(8,5,3,2,1)\}.
\end{align*}
But the assumption that $\la\notin H_3$ then gives
\[
\la\in\{(6,4,3,2,1),(7,5,3,2,1),(8,4,3,2,1)\}.
\]
The case $\la=(6,4,3,2,1)$ can be excluded by consulting decomposition matrices. If $\la=(7,5,3,2,1)$ or $(8,4,3,2,1)$ then $i=1$; but $\la^{-0}=(6,5,3,2)$ or $(8,4,3,2)$, so $\la^{-0}\notin H_0\cup H_1\cup H_2\cup H_3\cup\{\la\}$, contrary to assumption.
\case
Suppose $\la^{-i}=(3l-2\df4,\ 3,1)$ with $l\geq 1$. Since $\heps_0(\la^{-i})>0$, we deduce that $i=1$ and $a\leq\hphi_1(\la^{-1})=l$. In fact $a<l$, since if $a=l$ then $\la=(3l-1\df2)\sqcup(3)\in H_2$, contrary to assumption.

Since $\la^{-1}$ is restricted, $(\la^{-1})^\reg=\la^{-1}$. So
\[
\la^\reg=\tf_1^a((\la^{-1})^\reg)=(3l-2\df3a+1,3a-1\df2)\sqcup(3)
\]
and then $\heps_0(\la)=\eps_0(\la^\reg)=\max\{0,2(l-a)-2a-1\}$.

By assumption $a\geq 1$, so that $(l+1,2)$ must be added to $\la^{-1}$ to obtain $\la$ (otherwise $\la^\reg$ would not have the right form). In the first $l-1$ rows of $\la$ there are $a-1$ parts which are congruent to $2$ modulo $3$ and $l-a$ parts which are congruent to $1$. Now we give a lower bound for $\heps_0(\la)$: in each row $j\ls l-2$ for which $\la_j\equiv1\Md3$ there is a \srmbl $0$-node, and if in addition $\la_{j+1}\not=\la_j-2$ there are two \srmbl $0$-nodes. Hence
\[
\heps_0(\la)\gs\max\{l-a-1,2(l-a)-a-1\},
\]
with equality only if $\la_{l-1}=4$. So
\[
\max\{0,2(l-a)-2a-1\}=\eps_0(\la^\reg)=\heps_0(\la)\gs\{l-a-1,2(l-a)-a-1\}.
\]
This contradicts the assumption that $1\ls a<l$, except in the case where $a=l-1$ and
\[
\la=(3l-1\df8,\ 4,3,2).
\]
When $l=2$ this partition lies in $H_3$, so assume $l\gs3$, and let
\begin{align*}
\nu=\la^{+0}&=(3l+1\df10,\ 4,3,2,1).
\\
\mu&=(3l+1\df10,\ 9,1).
\end{align*}
Then $\mu^\reg=\nu^\reg$, and we claim that $\ddeg\mu<\ddeg\nu$. For $l=3$ this can be checked directly, and for $l\gs4$ it follows by induction using \cref{l419}.

So $\nu$ is not homogeneous, and hence neither is $\la$.
\case
Suppose $\la^{-i}=(3l-2\df1)+3\al$ for some partition $\al$ with $l(\al)\leq l$. Then $\heps_0(\la^{-i})>0$, so $i=1$ and $a\leq\hphi_1(\la^{-1})=l$. In fact we can assume $a<l$, since if $a=l$ then $\la=(3l-1\df2)+3\al\in H_0$.

Since all parts of $\la^{-1}$ are congruent to $1$ modulo $3$,
\begin{align*}
(\la^{-1})^\reg&=(\la^{-1}_1)^\reg+\dots+(\la^{-1}_l)^\reg\\
\intertext{and therefore}
\la^\reg&=\tf_1^a((\la^{-1})^\reg)=(\la^{-1}_1+1)^\reg+\dots+(\la^{-1}_a+1)^\reg+(\la^{-1}_{a+1})^\reg+\dots+(\la^{-1}_l)^\reg.
\end{align*}
It follows that
\[\eps_0(\la^\reg)=\max\{0,2(l-a)-\de-2a\}\]
with $\de=1$ if $\la^{-1}_l=1$ and $\de=0$ otherwise. On the other hand, since $\la$ is obtained from $\la^{-1}$ by adding nodes at the end of $a$ rows we can use reasoning similar to the previous case to obtain
\[\heps_0(\la)\geq l-a+\max\{0,l-2a-\de\}=\max\{l-a,2(l-a)-\de-a\}.\]

Now the fact that $\eps_0(\la^\reg)=\heps_0(\la)$ contradicts the assumption that $1\ls a<l$.

\medskip
We have now dealt with all cases where $\la^{-1}\neq\la$. So for the remaining cases we assume $\la^{-1}=\la$.

\case
Suppose $\la^{-1}=\la$ and $\la^{-0}=(3l-1\df5,\ 3,2)$ with $l\gs1$. In this case $a\ls\hphi_0(\la^{-0})=2l+1$ and $(\la^{-0})^\reg=\la^{-0}$.

First suppose $a=1$. Then $\la=\la^\reg=(3l-1\df5,\ 3,2,1)$, and the assumption that $\la^{-1}=\la$ means that $l=1$, so that $\la=(3,2,1)\in H_3$.

So we can assume $a\gs2$. Because $\la^\reg=\tf_0^a(\la^{-0})$, the nodes added to $\la^{-0}$ to obtain $\la$ must be $(l+2,1)$, $(l+1,3)$ and $a-2$ nodes in rows $1,\dots,l-1$. Since $\la$ is strict, the node $(l,4)$ must be added, so that in particular $a\gs3$. If $l=1$ then $\la=(4,3,1)\in H_2$, so assume from now on that $l\gs 2$. So if we let $a_r$ be the number of added nodes in row $r$ for $1\ls r\ls l-1$, then each $a_r\in\{0,1,2\}$, and $\la$ is determined by $(a_1,\dots,a_{l-1})$. The assumption that $\la^{-1}=\la$ means that there is no $1\ls r\ls l-2$ for which $a_r=0$ and $a_{r+1}<2$ (otherwise $\la$ would have a \srmbl $1$-node in row~$r$).

Now consider the partition $\la^{+1}$. This has a \sadbl $0$-node $(l+3,1)$; if it has any \srmbl $0$-nodes, these lie in shorter ladders, and so $\hep0{\la^{+1}}>\ep0{(\la^{+1})^\reg)}$ by \cref{dn0}, so that $\la^{+1}$, and hence $\la$, are not homogeneous. Note in particular that if there is $1\ls r\ls l-2$ for which $(a_r,a_{r+1})=(0,2)$, $(1,0)$ or $(1,1)$, then $\la^{+1}$ does have a \srmbl $0$-node (in row $r+1$, $r$ or $r$ respectively). So we can assume there is no such $r$.

So we assume that the tuple $(a_1,\dots,a_{l-1})$ contains no $0$s (except for possibly $a_{l-1}$), and does not contain two consecutive entries less than $2$.

\subcase

\edef\nucase{\subcaselabel}

Suppose there are at least two values of $r$ for which $a_r<2$. Let $r,s$ be the two smallest such values. Then
\begin{align*}
(\la_1,\dots,\la_{r-1})&=(3l+1\df3(l-r)+7),
\\
\la_r&=3(l-r)+3,
\\
(\la_{r+1},\dots,\la_{s-1})&=(3(l-r)+1\df3(l-s)+7),
\\
\la_s&=3(l-s)+3\text{ or }3(l-s)+2.
\end{align*}

Let
\[
\nu=((((\la^{+1})^{+0})^{+1})^{+0}\dots)^{+1},
\]
where ${}^{+1}$ appears $s-r$ times. As $r-s\geq 2$ we can calculate
\begin{align*}
(\nu_1,\dots,\nu_{r-1})&=(3(l+s-r)-1\df3(l+s-2r)+5),
\\
(\nu_r,\dots,\nu_{s-2})&=(3(l+s-2r)-1\df3(l-r)+5),
\\
\nu_{s-1}&=3(l-r)+3,
\\
\nu_s&\ls3(l-r)-1,
\end{align*}
so that $\nu$ has a \srmbl $0$-node $(s-1,3(l-r)+3)$.

\clam
If $a_{l-1}>0$ then $(l+s-r+2,1)$ is a \sadbl node of $\nu$.
\prof
Since $\la$ has length $l+2$, and each step of adding all \sadbl $0$-nodes can add at most one node in column $1$, the length of $\nu$ is at most $l+2+(s-r-1)$. So certainly $(l+s-r+2,1)$ is not a node of $\nu$. On the other hand, the assumption that $a_{l-1}>0$ means that $\la$ contains the partition $\si(l)$ from \cref{casefourlem}. So (by \cref{casefourlem}) $\nu$ contains the partition $\tau(l+s-r-1)$, and in particular contains the node $(l+s-r+1,2)$, so that $(l+s-r+2,1)$ is a \sadbl node of $\nu$.
\malc

\clam
If $a_{l-1}=0$ then $(l+s-r,3)$ is a \sadbl node of $\nu$.
\prof
In this case $(\la_{l-1},\la_l,\dots)=(5,4,3,1)$. So if we let $\xi$ be the partition obtained from $\la$ by adding all \sadbl $1$-nodes, then all \sadbl $0$-nodes, then all \sadbl $1$-nodes, then $(\xi_{l-1},\xi_l,\dots)=(8,5,3,2,1)$. In particular, $\xi'_3=l+1$. Since each step of adding all \sadbl $0$-nodes can add at most one node in column $3$, we get $\nu'_3\ls l+s-r-1$, so that $(l+s-r,3)$ is not a node of $\nu$. On the other hand, $\la$ contains the partition $\si(l-1)$, so by \cref{casefourlem}) $\nu$ contains the partition $\tau(l+s-r-2)$, and in particular contains the nodes $(l+s-r-2,5)$ and $(l+s-r-1,3)$, which means that $(l+s-r,3)$ is a \sadbl node of $\nu$.
\malc

In either case we see that $\nu$ has a \sadbl $0$-node in a longer ladder than the \srmbl node $(s-1,3(l-r)+3)$, so by \cref{dn0} $\nu$ is not homogeneous, and hence neither is $\la$.

\subcase

Suppose now that $(a_1,\dots,a_l)$ contains a single $0$, with the remaining entries being $2$s. Then $(a_1,\dots,a_{l-1})=(2,2,\dots,2,0)$, so that
\[
\la=(3l+1\df10,\ 5,4,3,1).
\]
If $l=2$ then $\la=(5,4,3,1)\in H_3$, so assume $l\gs3$. Now let
\[
\nu=\la^{+1}=(3l+2\df11,\ 5,4,3,2).
\]
Now we can calculate
\[
\nu^\reg=(3l-1,3l-2,\ 3l-4\df5,\ 3,2),
\]
giving
\[
(\nu^\reg)\ipi0=(3l+1,3l-2,3l-4,\ 3l-5\df4,\ 3,1).
\]
On the other hand,
\[
\nu^{+0}=(3l+4\df13,\ 7,4,3,2,1),
\]
giving
\[
(\nu^{+0})^\reg=(3l+1,3l-2,3l-3,\ 3l-5\df4,\ 2,1)\neq(\nu^\reg)\ipi0.
\]
So by \cref{T300921_2} $\nu$ is not homogeneous, and hence $\la$ is not homogeneous either.

\subcase

Finally consider the case where $(a_1,\dots,a_{l-1})$ contains a single $1$, and the rest are $2$s. Write $\la=\la_{(l,r)}$, where $a_r=1$ and $a_s=2$ for all $s\neq r$.

\subsubcase

First we suppose $r\ls l-2$ (so that in particular $l\gs3$). Then
\[
\la_{(l,r)}=(3l+1\df3(l-r)+7,\ 3(l-r)+3,\ 3(l-r)+1\df4,\ 3,1),
\]
and we obtain $\ddeg{\la_{(l,r)}}>\ddeg{\mu_{(l)}}$ and $\la_{(l,r)}^\reg=\mu_{(l)}^\reg$ by induction on $r$, using \cref{l419,deglem11}.

\subsubcase

We are left with the case where $r=l-1$, so that
\[
\la=(3l+1\df10,\ 6,4,3,1).
\]
For $l=2$ we use known decomposition numbers to show that $\la$ is not homogeneous, while for $l\gs3$ we can use \cref{deglem1}.

\case
\edef\hardcase{\caselabel}
Finally suppose that $\la^{-1}=\la$ and $\la^{-0}=(3l-1\df2)+3\al$, where $l(\al)\ls l$. We can assume $l\gs2$, since the case $l=1$ is included in \onepartcase. Now $\la^{-0}$ has two \sadbl $0$-nodes in each of rows $1,\dots,l$, and one in row $l+1$. We set $a_r=\la_r-\la^{-0}_r$ for $1\ls r\ls l$ and $a_{l+1}=\la_{l+1}+1$.

The assumption that $\la^{-1}=\la$ means that if $r\ls l$ with $a_r=0$, then $a_{r+1}=2$. Furthermore, if $a_l=0$, then $\al_l$ must be $0$.

If $a_1=\dots=a_l=2$, then $\la=(3l+1\df4)+3\al=(3l-2\df1)+3(\al+(1^l))$ or $(3l+1\df1)+3\al$, so $\la$ is \spc, contrary to assumption. So we assume there is some $r\ls l$ such that $a_r\ls1$.

If $r<s\ls l$ and $\al_r>\al_s$, then the \sadbl nodes of $\la^{-0}$ in row $r$ lie in a longer ladder than the \sadbl nodes in row $s$. So using \cref{dn0}, if $a_s>0$, then $a_r=2$. Similarly, if $a_{l+1}=2$ and $\al_r>0$, then $a_r=2$. In particular, if $\al_l>0$ then $a_{l+1}=1$ (since otherwise $a_1=\dots=a_l=2$, contrary to the last paragraph).

Combining these statements, we obtain $a_r=2$ for every $r$ for which $\al_r>\al_l$.
\subcase
\edef\arone{\subcaselabel}
Suppose $a_r=1$ for a single value $r$, and $a_s=2$ for all $s\neq r$. So
\[
\la=(3l+1\df1)+3\alpha-\delta_r,
\]
where $\delta_r$ is the composition which has a $1$ in position $r$ and $0$ elsewhere. The assumptions above mean that $\al_r=0$ and $a_{l+1}=2$. Now we use a dimension argument. Let $\mu=(3l+1\df4)+3\alpha$. Then $\mu^\reg=\la^\reg$ and  $\ddeg\la>\ddeg\mu$ by \cref{l419,deglem12} as $\al_r=0$. So $\la$ is not homogeneous.
\subcase
\edef\rpo{\subcaselabel}
Suppose that there exists $r$ such that $a_r=2$ and either $r=1$ or $a_{r-1}>0$.  Consider the partition $\la^{+1}$. We claim that this partition has a \sadbl $0$-node $(r,3(l-r+2+\al_r))$: this follows since $\la_r=3(l-r+2+\al_r)-2$, and $\la_{r-1}\gs3(l-r+2+\al_r)$ if $r\gs2$, and $\la_{r-2}\gs3(l-r+2+\al_r)+2$ if $r\gs3$.

If there is any $s\ls l$ for which $(a_s,a_{s+1})$ equals $(0,2)$, $(1,0)$ or $(1,1)$, then $\la^{+1}$ has a \srmbl $0$-node in either row $s$ or row $s+1$. Now the fact that $a_r=2$ and $a_s\ls1$ and the argument at the start of \hardcase{} mean that $\al_s\ls\al_r$, and therefore the \srmbl $0$-node in row $s$ or $s+1$ lies in a shorter ladder than the \sadbl $0$-node in row $r$. So by \cref{dn0} $\la^{+1}$ is not homogeneous, and hence neither is $\la$.

So we can assume that $a_s\gs1$ for all $s$, and that there are no two consecutive values of $s$ with $a_s=1$. As observed above, there is at least one $1\leq s\leq l$ with $a_s=1$. If there is exactly one $1\leq s\leq l+1$ for which $a_s=1$, then we are in case \arone, so assume instead that there are at least two values of $s$ for which $a_s=1$.

Now we define a sequence of partitions $\nu^{(1)},\nu^{(2)},\dots$ by
\[
\nu^{(1)}=\la^{+1},\qquad\nu^{(k+1)}=((\nu^{(k)})^{+0})^{+1}\text{ for }k\gs1.
\]
Observe that $\la^{+1}$ has no parts congruent to $1$ modulo $3$, and so by \cref{L110222} $\nu^{(k)}$ has no parts congruent to $1$ modulo $3$, for any $k$. Now we make the following observations, for $k\gs1$.
\clam
If $a_1=2$ or $a_{l+1}=2$ or the tuple $(a_1,\dots,a_{l+1})$ contains $k$ consecutive $2$s, then $\nu^{(k)}$ has a \sadbl $0$-node $(s,c)$ for which $3s+c\gs3(\al_l+l+k+1)$.
\prof
If $a_1=2$, then we can take $(s,c)=(1,3(l+k+\al_1))$. If $a_{l+1}=2$, then we can take $(s,c)=(l+k+1,1)$ by \cref{L110222_2} and induction on $k$.

Assume now that $k<r\ls l$ such that $a_{r-k}=1$ and $a_{r-k+1}=\dots=a_r=2$. Then $\al_{r-k}=\dots=\al_r=\al_l$, so we can take $(s,c)=(r,3(\al_l+l+k-r+1))$ by \cref{L110222_3}.
\malc

\clam
If there is $k\ls r\ls l$ such that $(a_{r-k+1},\dots,a_{r+1})=(1,2,2,\dots,2,1)$, then $\nu^{(k)}$ has a \srmbl $0$-node $(s,c)$ for which $3s+c=3(\al_l+l+k)$.
\prof
In this case $\al_{r-1}=\dots=\al_r=\al_l$. We can take $(s,c)=(r,3(\al_l+l-r+k))$ by \cref{L110222_3}, since
\[\nu^{(k)}_{r+1}\ls \nu^{(1)}_{r+1}+3(k-2)+2\ls 3(\al_l+l-r+k-2)+2.\]
\malc

As a result, \cref{dn0} shows that $\nu^{(k)}$ is not homogeneous for some $k$ (and hence $\la$ is not homogeneous) unless the tuple $(a_1,\dots,a_{l+1})$ has the form $(1,2^t,1,2^t,\dots,1,2^t,1)$ for some $t\gs1$. So assume this is the case. Now the fact that $a_1<2$ implies $\al_1=\al_l$. By \cref{C210322} we can then assume that $\al=(1^2)$ or $\varnothing$. In the first case $\la=(9,7)$, which can be ruled out looking at the known decomposition matrices. So we may assume that $\al=\varnothing$. If there are at least three values of $s$ for which $a_s=1$, then we use a degree argument: specifically, we claim that there is a partition $\mu$ such that $\mu\doms\la$, $\mu^\reg=\la^\reg$, $\mu_1=\la_1$ and $\ddeg\mu>\ddeg\la$, which will show that $\la$ is not homogeneous. We prove this by induction on the number of $1$s in the tuple $(a_1,\dots,a_{l+1})$. If there are at least four $1$s, then we can use \cref{l419} and induction, so assume $(a_1,\dots,a_{l+1})=(1,2^t,1,2^t,1)$. Then we use \cref{deglem9}.

We are left with the case where $(a_1,\dots,a_{l+1})=(1,2^{l-1},1)$. Suppose first that $l\gs6$. Then
\[
\nu^{(k)}=(6l-4\df3l+2,\ 3l,\ 3l-4\df2).
\]
Now by \cref{deglem10} $\nu^{(k)}$ is not homogeneous, and hence neither is $\la$.

So we are left with the cases $2\ls l\ls5$. If $l=2$ then $\la=(6,4)$, which can be dealt with using known decomposition matrices. For $3\ls l\ls5$, \cref{projcases} shows that $\la$ is not homogeneous.

\subcase
\edef\rind{\subcaselabel}
Suppose that $a_1<2$ and there is no $r\gs2$ for which $a_r=2$ and $a_{r-1}>0$. Suppose in addition that $l\gs3$ and there is $r\gs4$ such that $a_r\neq1$. Since $a_1<2$ we have $\al_1=\al_l$. Since $l\geq 3$ we can then assume from \cref{C210322} that $\al=\varnothing$. Now $\la$ is not homogeneous by \cref{0211deg}.
\subcase
If neither \rpo{} nor \rind{} applies, then the tuple $(a_1,\dots,a_{l+1})$ has the form $(1,1,\dots,1)$, $(0,2,1,1,\dots,1)$ or $(1,0,2,1,1,\dots,1)$. Again $\al_1=\al_l$, so by \cref{C210322} $\al=(1^2)$ or $\varnothing$. In the first case $\la\in\{(9,6),(8,7)\}$ (as $a_3=1$ in this case by assumption) which can both be ruled out looking at known decomposition matrices. So we may again assume that $\al=\varnothing$. These cases can be dealt with using \cref{deglem2,deglem4,deglem8}, except for the following small cases.
\begin{itemize}
\item
$(6,3)$, $(5,4)$ and $(8,7,3)$ can be shown to be not homogeneous using known decomposition matrices.
\item
$(11,10,6,3)$, $(14,13,9,6,3)$ and $(17,16,12,9,6,3)$ are not homogeneous by \cref{projcases}.
\end{itemize}

This completes the proof of \cref{T081021}.

\end{document}